\RequirePackage{amsmath}
\documentclass{svjour3}                     
\smartqed  
\usepackage{color}
\usepackage{fix-cm}
\usepackage{amsmath,amsfonts,amssymb,bm,mathtools}
\usepackage{graphicx,subfigure}
\usepackage{enumitem}
\usepackage{epstopdf}
\usepackage{empheq}
\usepackage{float}
\usepackage{bm}
\usepackage{fullpage}
\usepackage{placeins}
\allowdisplaybreaks

\def\subFV{\scriptscriptstyle{FV}}
\def\subFS{\scriptscriptstyle{FS}}
\def\subVS{\scriptscriptstyle{VS}}

\newtheorem{thm}{Theorem}

\journalname{Journal of Scientific Computing}
\begin{document}

\title{
Structure-preserving upwind Lagrange multiplier schemes for solid-state dewetting with a logarithmic Flory--Huggins potential}

\titlerunning{Structure-preserving schemes for simulating solid-state dewetting}

\author{Qiong-Ao Huang \and Ying-Wei Wang \and Cheng Yuan}


\institute{Qiong-Ao Huang\at
School of Mathematics and Statistics, Henan University, Kaifeng 475004, China\\
Center for Applied Mathematics of Henan Province, Henan University, Zhengzhou 450046, China\\
\email{huangqiongao@henu.edu.cn}
\and
Ying-Wei Wang\at
School of Mathematics and Statistics, Henan University, Kaifeng 475004, China\\
\email{wangyingwei@henu.edu.cn}
\and
Cheng Yuan\at
School of Artificial Intelligence, Wuhan University, Wuhan 430072, China\\
\email{yuancheng@whu.edu.cn}
}

\date{Received: date / Accepted: date}

\maketitle

\begin{abstract}
Phase-field simulations of solid-state dewetting based on polynomial potentials
suffer from spurious bulk-diffusion coarsening that contradicts the surface-diffusion-dominated
kinetics of the underlying physics.
To resolve this issue, we formulate a phase-field model with the logarithmic Flory--Huggins potential
for the degenerate Cahn--Hilliard equation. The logarithmic barrier intrinsically confines the phase variable to its physical range and thereby removes the spurious coarsening at the continuum level. The model also incorporates dynamic contact line boundary conditions
for the motion of the film--substrate--vapor triple junction, whose discrete treatment is a key difficulty addressed in this work.
A new fully discrete, structure-preserving scheme is developed by combining
a Lagrange multiplier approach with an upwind finite-volume discretization.
The scheme is rigorously proved to guarantee pointwise boundedness of the phase variable,
discrete mass conservation, and energy dissipation. These three properties are established without any artificial cut-off or projection step.
A dimensional-splitting technique is introduced to reduce computational cost,
and it is proved that all three structural properties are preserved exactly in every one-dimensional sweep.
A theoretical analysis of spontaneous film shrinking yields an explicit estimate
of the equilibrium radius contraction. To our knowledge, this provides the first explicit quantification of this spurious effect for film--substrate systems with moving contact lines, and it demonstrates that the logarithmic potential
exhibits weaker spurious shrinkage than its polynomial counterpart.
Numerical experiments confirm the theoretical predictions and demonstrate that,
when the temperature parameter in the logarithmic potential is small, the proposed
scheme eliminates spurious coarsening and pinch-off,
faithfully reproducing surface-diffusion-dominated interfacial dynamics
and morphological relaxation toward equilibrium island structures.


\end{abstract}

\section{Introduction}
\label{secint}

Solid-state dewetting refers to the spontaneous rupture and agglomeration of a thin solid film deposited on a substrate into discrete, island-like structures, driven by surface diffusion at elevated temperatures.
This phenomenon has been observed across a wide spectrum of material systems,
from semiconductors such as Si and Ge to metals including Au, Ni, and Co,
and carries considerable technological implications for thin-film solar cells~\cite{Danielson06}, sensor devices~\cite{Rath07},
catalytic growth of carbon nanotubes~\cite{Randolph07}, optoelectronics~\cite{Armelao06}, and semiconductor nanowires~\cite{Schmidt09}.
Fundamental understanding of the underlying pattern-formation mechanisms has advanced substantially over the past decades~\cite{Ye10,Ye10b,Ye11b,Wong00,Naffouti17},
as comprehensively surveyed in the reviews by Thompson~\cite{Thompson12} and by Leroy~\textit{et~al.}~\cite{Leroy16}.

At first glance, the morphological evolution of solid thin films resembles that of liquid films:
both may undergo retraction, pinch-off, and eventual equilibration into arrays of isolated particles.
The two processes, however, are governed by fundamentally distinct mass transport mechanisms~\cite{Craster09,Gennes85,Ren07,Ren10,Baumchen14}.
Liquid-film dewetting is driven by fluid flow and hydrodynamic instabilities,
whereas solid-state dewetting proceeds exclusively by surface diffusion along the film--vapor interface.
Consequently, the latter can be formulated as a moving-boundary problem for surface diffusion flow,
with the contact line (where the film, substrate, and vapor phases meet) migrating as the interface evolves~\cite{Jiang12,Wang15,Jiang16,Jiang20,Jiang18a}.

Two widely used mathematical frameworks for modeling solid-state dewetting are sharp-interface models and phase-field models.
The first sharp-interface formulation was proposed by Srolovitz and Safran~\cite{Srolovitz86},
who employed geometric flow theory to study hole growth under isotropic surface energy and cylindrical symmetry.
Subsequent efforts extended this framework in several directions: Wong~\textit{et~al.}~\cite{Wong00} and Du~\textit{et~al.}~\cite{Du10} introduced marker-particle methods for two- and three-dimensional simulations,
while a systematic energy-variational derivation of evolution equations for open curves with weakly and strongly anisotropic surface energies was carried out by several groups~\cite{Wang15,Jiang16,Jiang18a},
with extensions to axisymmetric~\cite{Zhao19} and fully three-dimensional geometries~\cite{Jiang20}.
Sharp-interface models provide a precise description of interfacial kinematics,
yet they face considerable practical difficulties in handling contact line migration and topological changes such as pinch-off and coalescence. Phase-field models circumvent these difficulties by replacing sharp interfaces with diffuse transition layers,
thereby capturing topological events naturally and extending readily to higher dimensions and complex geometries~\cite{Jiang12,Dziwnik17,Garcke22}.
Within this framework, the total free energy of the film--substrate system is expressed as a Ginzburg--Landau-type functional,
and the governing equations follow as an $H^{-1}$ gradient flow, yielding a degenerate Cahn--Hilliard equation.

A distinct challenge arises at the contact line.
In hydrodynamic problems involving moving contact lines,
the classical no-slip boundary condition leads to a non-integrable stress singularity,
a difficulty known as the contact line paradox~\cite{Dussan79}.
Experiments, atomistic simulations, and studies of crystal interface migration all indicate that the dynamic contact angle relaxes gradually toward its equilibrium (Young) value~\cite{Mattissen05,Qian03}.
To regularize this behavior, Qian~\textit{et~al.}~\cite{Qian03,Qian06} introduced a dynamic contact line boundary condition
that relates the contact line velocity to the local chemical potential gradient,
effectively allowing the contact line to move via a dissipative mechanism.
This boundary condition has been incorporated into phase-field models for solid-state dewetting~\cite{Huang19b}.

At the continuum level, the phase-field evolution equations, supplemented with appropriate boundary conditions,
satisfy two fundamental thermodynamic constraints: energy dissipation and mass conservation.
Preserving these properties, together with the pointwise boundedness of the phase variable,
at the discrete numerical level is essential for eliminating unphysical artifacts and constitutes the central objective in the construction of structure-preserving schemes. The simultaneous preservation of these three properties under dynamic contact line boundary conditions is the main algorithmic challenge, and it is precisely this challenge that the present scheme resolves.
To this end, several effective approaches have been proposed for the phase-field modeling of solid-state dewetting,
including convex splitting~\cite{Backofen19}, the invariant energy quadratization (IEQ) method~\cite{Huang19b},
and the scalar auxiliary variable (SAV) method~\cite{Chen20}.
More recently, Huang~\textit{et~al.}~\cite{Huang23b} developed an upwind Lagrange multiplier scheme that combines upwind numerical fluxes with a Lagrange multiplier reformulation of the energy functional.
By using the upwind flux to discretize the degenerate mobility,
the scheme naturally enforces the boundedness of the phase variable without artificial cut-off or truncation,
while retaining the energy stability inherited from the Lagrange multiplier framework.
This approach was originally formulated for homogeneous Neumann boundary conditions;
extending it to dynamic contact line boundary conditions,
which are indispensable for capturing the physics of the film--substrate--vapor triple junction,
remains an open problem that the present work addresses.

A second, more subtle numerical pathology plagues conventional phase-field simulations of solid-state dewetting.
When the polynomial double-well potential $F_{pol}(\phi)=\frac{1}{4}(1-\phi^2)^2$ is paired with a degenerate mobility $M(\phi)=(1-\phi^2)^k$, matched asymptotic analysis formally recovers surface diffusion in the sharp-interface limit for $k\ge 2$~\cite{Lee16}. In practice, however, the polynomial potential admits values of $\phi$ outside the physically admissible interval $[-1,1]$,
and the resulting bulk diffusion drives spurious mass transfer between well-separated domains of the same phase:
smaller features are progressively absorbed by larger ones, a phenomenon known as bulk-diffusion coarsening~\cite{Pesce21,Bretin22}. Cahn~\textit{et~al.}~\cite{Cahn96} proved that replacing $F_{pol}$ with the logarithmic Flory--Huggins potential eliminates this artifact: the logarithmic barrier at $\phi=\pm 1$ intrinsically confines the phase variable to $|\phi|\le\beta_{\theta}<1$,
and for the temperature scaling $\theta=\mathcal{O}(\varepsilon^{\alpha})$ with $\alpha>0$,
the sharp-interface limit remains surface diffusion.
Despite this theoretical guarantee, a fully discrete, structure-preserving scheme that pairs the logarithmic Flory--Huggins potential with an upwind Lagrange multiplier discretization and dynamic contact line boundary conditions has not, to our knowledge, been developed or analyzed.

In this paper, we close this gap. Our principal contributions are as follows:

\begin{itemize}
    \item 
We formulate a phase-field model for solid-state dewetting in which the logarithmic Flory--Huggins potential replaces the conventional polynomial potential, thereby suppressing spurious bulk-diffusion coarsening at the PDE level. Unlike the polynomial potential, the logarithmic barrier intrinsically enforces $|\phi|\le\beta_{\theta}<1$ and thus removes this artifact at its physical origin.

    \item 
We construct a fully discrete, structure-preserving scheme for the degenerate Cahn--Hilliard equation with dynamic contact line boundary conditions, combining a Lagrange multiplier approach with an upwind finite-volume spatial discretization. A novel feature is the discrete treatment of the moving contact line, which couples the bulk evolution to the wall energy without violating the structural properties. The scheme is rigorously proved to ensure pointwise boundedness of the phase variable, discrete mass conservation, and energy dissipation.

    \item 
We introduce a dimensional-splitting technique that decomposes each time step into alternating one-dimensional sweeps, substantially reducing the computational cost of multidimensional simulations. We prove that boundedness, mass conservation, and energy dissipation are all preserved exactly by the split scheme, so that the dimensionality reduction does not come at the expense of structure preservation.

    \item 
We perform a theoretical analysis of spontaneous film shrinking induced by the diffuse-interface approximation and derive an explicit formula for the equilibrium radius contraction. This yields a quantitative, physically interpretable prediction of the spurious area loss, and it establishes that the logarithmic potential suffers markedly less from this spurious effect than its polynomial counterpart.

\end{itemize}
Extensive numerical experiments, spanning spontaneous shrinkage, coarsening suppression, contact line dynamics, morphological equilibration under varying wettability, and pinch-off, confirm the theoretical results and demonstrate the efficacy of the proposed framework.

The remainder of this paper is organized as follows.
Section~2 introduces the phase-field model with dynamic contact line boundary conditions and presents the theoretical analysis of spontaneous film shrinking.
Section~3 develops the fully discrete upwind-Lagrange multiplier scheme and proves its structure-preserving properties.
Section~4 introduces a dimensional-splitting technique that reduces the computational cost of multidimensional simulations,
with corresponding proofs that all three structural properties are preserved.
Section~5 presents a comprehensive suite of numerical experiments.
Section~6 concludes with a summary and outlook.

\section{Phase-field model}
\label{sec2}

We first formulate the coupled energy structure that the numerical method will preserve. Starting from the bulk and wall energies, we derive the conserved bulk evolution and dynamic contact line condition. We then examine why these structural laws can coexist with geometric area loss by estimating spontaneous shrinkage in a simplified circular-segment geometry.

\begin{figure}[!htp]
\centering
\includegraphics[width=8.5cm]{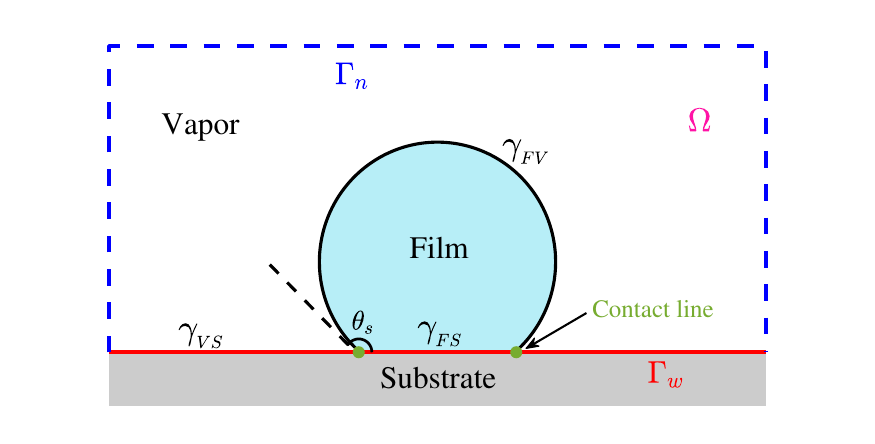}
\includegraphics[width=6.0cm]{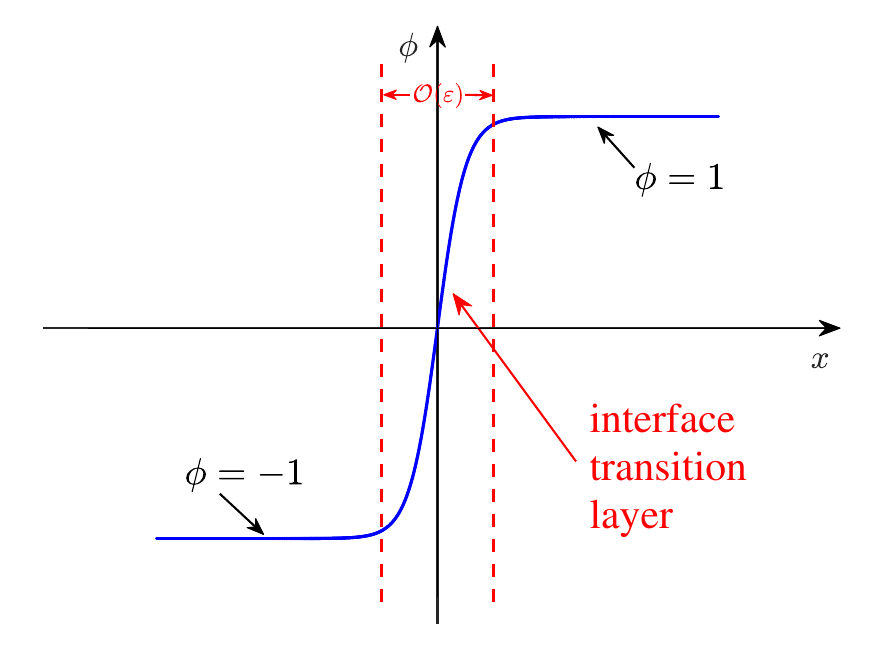}
\caption{
Left: film--substrate geometry and interfacial energy densities; $\theta_s$ is the Young contact angle. Right: a schematic diffuse-interface profile. The labels $\pm1$ illustrate limiting pure states; the logarithmic-potential minima are at $\pm\beta_\theta$.}\label{figdomain}
\end{figure}

\subsection{Free energy}

In the sharp-interface framework, the total interfacial free energy $W$ of the thin film/substrate system
(as illustrated in Fig.~\ref{figdomain}) is given by~\cite{Wang15,Jiang16,Jiang18a,Jiang20}
\begin{eqnarray}
W=W_{_{FV}}+W_{_{W}}
=\gamma_{_{\subFV}}|\Gamma_{_{\subFV}}|+\underbrace{\gamma_{_{\subFS}}|\Gamma_{_{\subFS}}|+
\gamma_{_{\subVS}}|\Gamma_{_{\subVS}}|}\limits_{\textbf{Wall Energy}},\label{eqo1}
\end{eqnarray}
\begingroup\emergencystretch=1em

where $\gamma_{{\subFV}}$, $\gamma_{{\subFS}}$, and $\gamma_{{\subVS}}$ denote the surface energy densities of the film/vapor, film/substrate, and vapor/substrate interfaces, respectively, while $|\Gamma_{_{\subFV}}|$, $|\Gamma_{_{\subFS}}|$, and $|\Gamma_{_{\subVS}}|$ represent the corresponding lengths (in 2D) or areas (in 3D) of these interfaces. The first term in \eqref{eqo1} is the film--vapor energy; the remaining two terms constitute the wall energy, with constant substrate interfacial energy densities.
\par\endgroup

While the sharp-interface formulation~\eqref{eqo1} provides a clear physical picture,
it is ill-suited for handling topological changes such as pinch-off and coalescence.
We therefore recast the total energy in a diffuse-interface framework,
where the sharp interfaces are replaced by thin transition layers of width $\varepsilon$ (see Fig.~\ref{figdomain}).
Within the phase-field framework, the corresponding total free energy for solid-state dewetting
is expressed as~\cite{Jiang12,Huang19b,Dziwnik17,Garcke22}
\begin{equation}		
W^{\varepsilon}[\phi(\bm{x},t)]=W_{_{FV}}^{\varepsilon}+W_{_{W}}^{\varepsilon}
=\int_{\Omega}f_{_{FV}}(\phi)\mathrm{d}\bm{x}+\int_{\Gamma_{w}}f_{_{W}}(\phi)\mathrm{d}s,\quad
(\bm{x},t)\in\Omega\times[0,T],\label{eqo2}
\end{equation}
where $\Omega\subset\mathbb{R}^{d}$ is an open bounded domain

whose boundary consists of a non-substrate part $\Gamma_{n}$ and a substrate part $\Gamma_{w}$,
$W_{_{FV}}^{\varepsilon}$ represents the combined energy of the thin film and vapor phases,
$W_{_{W}}^{\varepsilon}$ denotes the wall energy,
and $f_{_{FV}}(\phi)$ and $f_{_{W}}(\phi)$ are the corresponding energy density functions.

The film/vapor phase energy density $f_{_{FV}}$ is defined as
\begin{equation}
f_{_{FV}}(\phi) \triangleq\lambda_{m}\left(\frac{\varepsilon}{2}|\nabla \phi|^{2}+\frac{1}{\varepsilon} F(\phi)\right),\label{eqo3}
\end{equation}
where $0<\varepsilon\ll 1$ is a small parameter controlling the interfacial width, $F(\phi)$ is a double-well potential,
and $\lambda_{m}=\gamma_{_{FV}}/c_{_{F}}$ represents the mixing energy density.
Here, $c_{_{F}}$ is an interfacial constant that depends on the specific form of $F(\phi)$, as detailed below.
A typical thermodynamically consistent form of $F(\phi)$ is given by the logarithmic Flory--Huggins potential~\cite{Cahn96,Huang23b}:
\begin{equation}
F(\phi)=F_{log}(\phi)-F_{log}(\beta_{_{\theta}}),\label{eqo4o}
\end{equation}
with
\begin{equation}
F_{log}(\phi;\theta)=\frac{\theta}{2}\big[(1+\phi)\ln(1+\phi)+(1-\phi)\ln(1-\phi)\big]+\frac{1}{2}(1-\phi^{2}),\quad 0<\theta<1,\label{eqo4}
\end{equation}
\begingroup\emergencystretch=1em

where $\theta$ is a dimensionless temperature parameter.
It can be verified that $F_{log}(\phi;\theta)$ exhibits a double-well structure, with two minima located at $\pm\beta_{_{\theta}}\in(-1,1)$,
where $\beta_{_{\theta}}$ is the positive root of the equation
$F_{log}'(\chi)=\frac{\theta}{2}\ln\frac{1+\chi}{1-\chi}-\chi=0$,
and satisfies $\beta_{_{\theta}}\to 1$ as $\theta\to 0^{+}$.
To ensure that the energy functional $W_{_{FV}}^{\varepsilon}$ with energy density $f_{_{FV}}$ defined in Eq.~\eqref{eqo3}
$\Gamma$-converges to the sharp-interface energy $\gamma_{_{FV}}|\Gamma_{_{FV}}|$,
the interfacial constant must be chosen as~\cite{Garcke22,Modica77,Modica87}:
\par\endgroup
\begin{equation}
c_{_{F_{log}}}=\int_{-\beta_{_{\theta}}}^{\beta_{_{\theta}}}\sqrt{2F(\chi)}\mathrm{d}\chi
=\int_{-\beta_{_{\theta}}}^{\beta_{_{\theta}}}\sqrt{2\big(F_{log}(\chi)-F_{log}(\beta_{_{\theta}})\big)}\mathrm{d}\chi.\label{eqo5}
\end{equation}

The wall energy $W_{_W}^{\varepsilon}$, expressed via the wall energy density $f_{_{W}}$, must satisfy the following physical constraints:
$f_{_{W}}=\gamma_{_{\subVS}}$ and $f'_{_{W}}=0$ when $\phi=-\beta_{_{\theta}}$ (pure vapor phase in contact with the substrate),
and $f_{_{W}}=\gamma_{_{\subFS}}$ and $f'_{_{W}}=0$ when $\phi=\beta_{_{\theta}}$ (pure film phase in contact with the substrate).
The simplest polynomial interpolant satisfying these four conditions is the cubic function:
\begin{eqnarray}
f_{_{W}}(\phi)=\frac{\gamma_{_{\subVS}}-\gamma_{_{\subFS}}}{4\beta_{_{\theta}}^{3}}(\phi^{3}-3\beta_{_{\theta}}^{2}\phi)
  +\frac{\gamma_{_{\subVS}}+\gamma_{_{\subFS}}}{2}.\label{eqo6}
\end{eqnarray}

For comparison, we use the standard quartic double-well potential~\cite{Jiang12,Huang19b,Dziwnik17,Garcke22}:
\begin{equation}
F_{pol}(\phi)=\frac{1}{4}(1-\phi^{2})^{2},\label{eqo7}
\end{equation}
which exhibits a double-well structure with minima at $\pm 1$.
The corresponding interfacial constant is~\cite{Jiang12,Garcke22}:
\begin{equation}
c_{_{F_{pol}}}=\int_{-1}^{1}\sqrt{2F_{pol}(\chi)}\mathrm{d}\chi=\frac{2\sqrt{2}}{3}.\label{eqcFb}
\end{equation}
In this case, the wall energy density $f_{_{W}}(\phi)$ takes the same form as \eqref{eqo6} with $\beta_{_{\theta}}$ replaced by $1$.

\subsection{Governing equations}

For convenience, we normalize the total free energy~\eqref{eqo2} by the factor $\varepsilon/\lambda_{m}$ and drop an additive constant that does not affect the gradient flow dynamics, yielding
\begin{equation}
W(\phi)=\int_{\Omega}\left(\frac{\varepsilon^2}{2}|\nabla \phi|^{2}+ F(\phi)\right)\mathrm{d}\bm{x}
  +\int_{\Gamma_{w}}g(\phi)\mathrm{d} s, \quad\text{where}\quad
g(\phi)\triangleq\frac{\varepsilon c_{_{F}}\cos\theta_{s}}{4\beta_{_{\theta}}^{3}}(\phi^{3}-3\beta_{_{\theta}}^{2}\phi), \label{eqo9}
\end{equation}
$F(\phi)$ is defined by \eqref{eqo4o} or \eqref{eqo7},
and $\theta_{s}\in[0,\pi]$ is the prescribed contact angle determined by the Young equation, i.e., $\cos\theta_{s}=\frac{\gamma_{_{VS}}-\gamma_{_{FS}}}{\gamma_{_{FV}}}$.

To derive the governing equations from the energy functional $W(\phi)$ in~\eqref{eqo9},
we compute its first variation. For any smooth test function $\psi$,
\begin{align}
\left.\frac{\mathrm{d} W[\phi+r \psi]}{\mathrm{d} r}\right|_{r = 0}
&=\left.\frac{\mathrm{d}}{\mathrm{d} r}\left[\int_{\Omega}\left(\frac{\varepsilon^2}{2}|\nabla(\phi+r \psi)|^{2}+ F(\phi+r \psi)\right) \mathrm{d} \bm{x}+\int_{\Gamma_{w}} g(\phi+r \psi) \mathrm{d}s\right]\right|_{r = 0} \nonumber\\
&=\int_{\Omega}\left(\varepsilon^2 \nabla \phi \cdot \nabla \psi+ F'(\phi) \psi\right) \mathrm{d} \boldsymbol{x}+\int_{\Gamma_{w}} g'(\phi) \psi \mathrm{d} s\nonumber\\
&=\int_{\Omega}\left(-\varepsilon^2 \Delta \phi+ F'(\phi)\right) \psi \mathrm{d} \bm{x}+\int_{\Gamma_{n}}(\varepsilon^2 \nabla \phi \cdot \bm{n}) \psi \mathrm{d} s+\int_{\Gamma_{w}}\left(\varepsilon^2 \nabla \phi \cdot \bm{n}+g'(\phi)\right) \psi \mathrm{d}s, \label{eqo10}
\end{align}
where $\bm{n}$ is the outward unit normal vector on $\partial\Omega$.

By taking the $H^{-1}$-gradient flow of the energy functional $W(\phi)$
with respect to the order parameter, we obtain the following
Cahn--Hilliard-type equation with degenerate mobility for solid-state dewetting:
\begin{equation}
\left\{\begin{array}{l}
\frac{\partial \phi}{\partial t}=-\nabla\cdot\bm{J},\\[1mm]
\bm{J}=-M(\phi) \nabla \mu,\\[1mm]
\mu=-\varepsilon^2 \Delta \phi+ F'(\phi),
\end{array}\right.\quad\text{in}~\Omega\times(0,T],\label{eqo11}
\end{equation}
subject to the following dynamic contact line boundary conditions
\begin{align}
\begin{cases}
\frac{\partial \phi}{\partial t}=-\kappa\left(\varepsilon^2 \frac{\partial \phi}{\partial \boldsymbol{n}}+g'(\phi)\right), \quad &\frac{\partial \mu}{\partial \boldsymbol{n}} = 0,\quad \text{on}~\Gamma_{w}, \\[1mm]
\varepsilon^2 \frac{\partial \phi}{\partial \boldsymbol{n}}=0, \quad 	&\frac{\partial \mu}{\partial \boldsymbol{n}} = 0,\quad \text{on} ~\Gamma_{n},
\end{cases} \label{eqo12}
\end{align}
where $\bm{J}$ is the mass flux, $\mu$ is the chemical potential, $M(\phi)\ge 0$ is the diffusion mobility,
$\kappa>0$ is the contact line mobility and $\bm{n}$ is the unit normal vector pointing outward from $\partial\Omega$.

The mobility determines the transport kinetics of \eqref{eqo11}--\eqref{eqo12} while leaving the energy functional unchanged. With constant mobility, the sharp-interface limit is the bulk-diffusion-driven Mullins--Sekerka problem~\cite{Pego89,Mullins63,Dai16b,Alikakos94}. A mobility that becomes small in the bulk phases reduces bulk transport, although the limiting law also depends on the potential and parameter scaling~\cite{Cahn94,Lee15}. We consider the family
\begin{equation}
M(\phi)=(1-\phi^{2})^{k},\quad k=1,2,3,\cdots,\label{eqo13}
\end{equation}
widely adopted in phase-field studies~\cite{Elliott96,Dai16,Pesce21}.

For the polynomial potential, matched asymptotic analysis with $k=2$ recovers surface diffusion~\cite{Lee16}, but finite-width computations can still exhibit bulk-transport artifacts~\cite{Pesce21,Bretin22}. For the logarithmic potential and $k=1$, Cahn~\textit{et~al.}~\cite{Cahn96} formally derived surface diffusion under the scaling $\theta=\mathcal{O}(\varepsilon^{\alpha})$, $\alpha>0$. The singular derivative at $\phi=\pm1$ must be distinguished from the locations $\pm\beta_{\theta}$ of the potential minima. The latter are equilibrium bulk values, not a bound established by the discrete analysis below. Indeed, $M(\pm\beta_{\theta})>0$ at finite temperature, so reduced coarsening is a low-temperature behavior rather than a consequence of degeneracy at these minima. The numerical experiments therefore assess finite-width effects as well as the structural laws derived next.

A direct calculation shows that the total free energy $W(t)$ defined in~\eqref{eqo9} is dissipated during the evolution:
\begin{align}
\frac{\mathrm{d}W(t)}{\mathrm{d}t}
&=\int_{\Omega}\left(\varepsilon^{2}\nabla\phi\cdot\nabla\phi_{t}+F'(\phi)\phi_{t}\right)\mathrm{d}\bm{x}
   +\int_{\Gamma_{w}}g'(\phi)\phi_{t}\,\mathrm{d}s\nonumber\\
&=\int_{\Omega}\mu\phi_{t}\,\mathrm{d}\bm{x}+\int_{\partial\Omega}\varepsilon^{2}\frac{\partial\phi}{\partial\bm{n}}\phi_{t}\,\mathrm{d}s
   +\int_{\Gamma_{w}}g'(\phi)\phi_{t}\,\mathrm{d}s\nonumber\\
&=\int_{\Omega}\mu\nabla\cdot(M(\phi)\nabla\mu)\mathrm{d}\bm{x}
   +\int_{\Gamma_{w}}\left[\varepsilon^{2}\frac{\partial\phi}{\partial\bm{n}}+g'(\phi)\right]\phi_{t}\,\mathrm{d}s\nonumber\\
&=-\int_{\Omega}M(\phi)|\nabla\mu|^{2}\mathrm{d}\bm{x}-\frac{1}{\kappa}\int_{\Gamma_{w}}(\phi_{t})^2\,\mathrm{d}s\le 0,\label{ener}
\end{align}
where the boundary conditions \eqref{eqo12} further imply the conservation of the total mass $m(t)$:
\begin{eqnarray}
\frac{\mathrm{d}m(t)}{\mathrm{d}t}=\frac{\mathrm{d}}{\mathrm{d}t}\int_{\Omega}\phi\,\mathrm{d}\bm{x}=\int_{\Omega}\phi_{t}\,\mathrm{d}\bm{x}
 =\int_{\Omega}\nabla\cdot(M(\phi)\nabla\mu)\mathrm{d}\bm{x}
 =\int_{\partial\Omega}M(\phi)\frac{\partial\mu}{\partial\bm{n}}\mathrm{d}s=0.\label{mass}
\end{eqnarray}

Our numerical objective is to preserve the mass and energy laws while enforcing $|\phi_{i,j}|<1$ through the discrete fluxes, without a separate projection of the computed solution.

\subsection{Spontaneous shrinkage}

\begin{figure}[!htp]
\centering
\includegraphics[width=8.0cm]{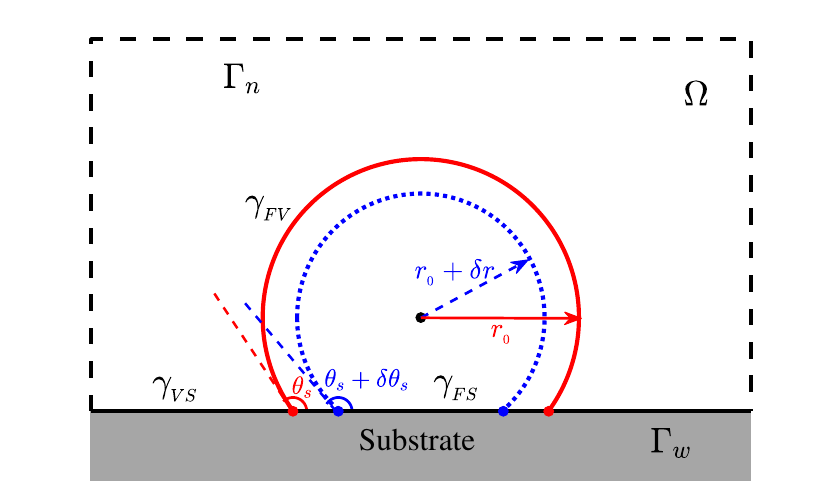}
\includegraphics[width=6.5cm]{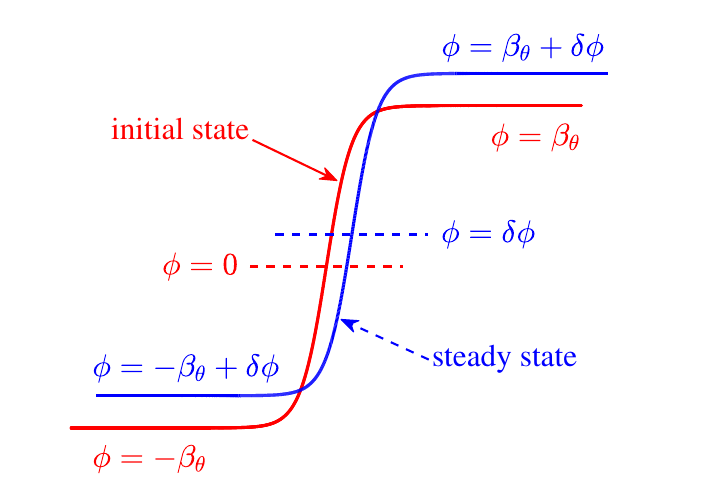}
\caption{
Shrinkage ansatz: (left) changes in radius and apparent contact angle at fixed circle-center height; (right) a common shift of the phase variable in the two bulk phases.}\label{figshrink}
\end{figure}

The mass law derived above constrains the integral of the phase variable, but it does not imply exact conservation of the geometric area enclosed by its zero contour. A film can therefore shrink geometrically while the diffuse-interface model still conserves mass. To quantify this distinction, we estimate shrinkage for a circular film segment on a substrate, extending the energetic argument for isolated drops in Yue~\textit{et~al.}~\cite{Yue07} to include wall energy. The calculation describes a restricted family of configurations and provides a leading-order estimate rather than a general equilibrium theorem.

\begingroup\emergencystretch=1em

We consider a thin-film/substrate system within a two-dimensional domain $\Omega$ of area $S=|\Omega|$
and substrate length $L=|\Gamma_{w}|$.
The initial film profile forms a circular segment (as illustrated in Fig.~\ref{figshrink}),
characterized by a Young contact angle $\theta_{s}\in(0,\pi)$ and radius $r_{_{0}}>0$.
The corresponding arc length between the two contact points is given by $L_{arc}=2r_{_{0}}\theta_{s}$,
the chord length by $L_{cho}=2r_{_{0}}\sin\theta_{s}$, and the film area by $S_{film}=\frac{1}{2}(2\theta_{s}-\sin2\theta_{s})r_{_{0}}^{2}$.
\par\endgroup

Let $\delta r$ denote a small change in radius and $\delta\theta_{s}$ the associated change in the geometric contact angle. We assume $\varepsilon\ll r_{_{0}}$ and approximate the bulk response by the same spatially uniform shift $\delta\phi$ in both phases. This is a leading-order ansatz motivated by the equal curvatures of the symmetric potential at its two minima and by a uniform equilibrium chemical potential. The symbol $\delta\theta_{s}$ describes a change in the apparent angle; the material Young angle remains fixed in the wall energy.

As an additional geometric assumption, we hold the vertical position of the circle center fixed while varying the radius. The circular-segment shape alone does not impose this constraint. For the family of configurations shown in Fig.~\ref{figshrink}, it gives
\begin{equation}
(r_{_{0}}+\delta r)\sin\big(\theta_{s}+\delta\theta_{s}-\frac{\pi}{2}\big)=r_{_{0}}\sin\big(\theta_{s}-\frac{\pi}{2}\big),\label{eqt1}
\end{equation}
or
\begin{equation}
\delta\theta_{s}\approx \frac{\cot\theta_{s}}{r_{_{0}}}\delta r.\label{eqt2}
\end{equation}
For this estimate and the subsequent analysis, we retain only the leading-order contributions.

Under the common bulk-shift ansatz, the deviations $\delta\phi$ from $\phi=\pm\beta_{_{\theta}}$ are equal in the two phases.
Enforcing the mass constraint given by \eqref{mass} then leads to the following relation between $\delta\phi$ and $\delta r$:
\begin{align}
m(t)=&\int_{\Omega}\phi\,\mathrm{d}\bm{x}\approx
 \beta_{_{\theta}}\cdot\frac{1}{2}(2\theta_{s}-\sin2\theta_{s})r_{_{0}}^{2}
 +(-\beta_{_{\theta}})\cdot\Big[S-\frac{1}{2}(2\theta_{s}-\sin2\theta_{s})r_{_{0}}^{2}\Big]\nonumber\\
\approx&\,(\beta_{_{\theta}}+\delta\phi)\cdot
          \frac{1}{2}\big[2(\theta_{s}+\delta\theta_{s})-\sin(2(\theta_{s}+\delta\theta_{s}))\big](r_{_{0}}+\delta r)^{2}\nonumber\\
&+(-\beta_{_{\theta}}+\delta\phi)
     \cdot\Big\{S-\frac{1}{2}\big[2(\theta_{s}+\delta\theta_{s})-\sin(2(\theta_{s}+\delta\theta_{s}))\big](r_{_{0}}+\delta r)^{2}\Big\},\label{eqt3o}
\end{align}
or
\begin{align}
\delta\phi&\approx\frac{\beta_{_{\theta}}(2\theta_{s}-\sin2\theta_{s})r_{_{0}}^{2}
  -\beta_{_{\theta}}\big[2(\theta_{s}+\delta\theta_{s})-\sin(2(\theta_{s}+\delta\theta_{s}))\big](r_{_{0}}+\delta r)^{2}}{S}\nonumber\\
&\approx-\frac{\beta_{_{\theta}}}{S}\big[2r_{_{0}}(2\theta_{s}-\sin2\theta_{s})\delta r+4r_{_{0}}^{2}\sin^{2}\theta_{s}\cdot\delta\theta_{s}\big]
\approx-\frac{4\beta_{_{\theta}}\theta_{s}r_{_{0}}}{S}\delta r.\label{eqt3}
\end{align}

We next analyze the variation in the energy, which comprises the changes in the interfacial energies
at the film--vapor, film--substrate, and vapor--substrate interfaces (denoted as $\delta W_{int}$),
as well as the change in bulk energy (denoted as $\delta W_{bulk}$).
The change in interfacial length, in turn, alters the interfacial energy:
\begin{align}
\delta W_{int}&\approx \gamma_{_{FV}}\delta L_{arc}+\gamma_{_{FS}}\delta L_{cho}+\gamma_{_{VS}}\delta (L-L_{cho})
=\gamma_{_{FV}}\left(\delta L_{arc}-\cos\theta_{s}\cdot\delta L_{cho}\right)\nonumber\\
&=\gamma_{_{FV}}\Big\{\big[2(r_{_{0}}+\delta r)(\theta_{s}+\delta\theta_{s})-2r_{_{0}}\theta_{s}\big]
  -\cos\theta_{s}\cdot\big[2(r_{_{0}}+\delta r)\sin(\theta_{s}+\delta\theta_{s})-2r_{_{0}}\sin\theta_{s}\big]\Big\}\nonumber\\
&\approx \gamma_{_{FV}}\big[(2\theta_{s}-\sin2\theta_{s})\delta r+2r_{_{0}}\sin^{2}\theta_{s}\cdot\delta\theta_{s}\big]
\approx 2\gamma_{_{FV}}\theta_{s}\delta r, \label{eqt4}
\end{align}
where the variations in $\gamma_{_{FV}}$, $\gamma_{_{FS}}$ and $\gamma_{_{VS}}$ have been neglected,
as their changes are of $\mathcal{O}(\delta\phi^{2})$.

The corresponding bulk-energy change is obtained by expanding about the minima $\pm\beta_{_{\theta}}$:
\begin{align}
\delta W_{bulk}\approx \frac{\lambda_{m}}{\varepsilon}\int_{\Omega}\delta F(\phi)\mathrm{d}\bm{x}
\approx\frac{\lambda_{m}SF''(\beta_{_{\theta}})}{2\varepsilon}\delta\phi^{2}
=\frac{\gamma_{_{FV}}SF''(\beta_{_{\theta}})}{2\varepsilon c_{_{F}}}\delta\phi^{2}
\approx\frac{8\gamma_{_{FV}}\beta_{_{\theta}}^{2}F''(\beta_{_{\theta}})\theta_{s}^{2}r_{_{0}}^{2}}{\varepsilon c_{_{F}}S}\delta r^{2}.\label{eqt5}
\end{align}

Combining Eqs.~\eqref{eqt4}-\eqref{eqt5}, we write the variation of the total free energy as:
\begin{equation}
\delta W=\delta W_{int}+\delta W_{bulk}\approx \gamma_{_{FV}}
\left(2\theta_{s}\delta r+\frac{8\beta_{_{\theta}}^{2}F''(\beta_{_{\theta}})\theta_{s}^{2}r_{_{0}}^{2}}{\varepsilon c_{_{F}}S}\delta r^{2}\right).
\label{eqt6}
\end{equation}
Given that $\frac{\partial(\delta W)}{\partial(\delta r)}|_{\delta r=0}=2\theta_{s}\gamma_{_{FV}}>0$
holds for $\theta_{s}\in (0,\pi)$, the total energy decreases when the film shrinks ($\delta r <0$).
The physical basis for this behavior lies in the relative scaling of energy contributions:
while the increase in bulk energy is $\mathcal{O}(\delta r^{2})$,
the reduction in interfacial energy scales as $\mathcal{O}(\delta r)$, making shrinkage energetically favorable.
Moreover, the condition $\frac{\partial(\delta W)}{\partial(\delta r)}=0$ characterizes the energy-minimizing state,

which gives the radius-contraction estimate within this restricted family:
\begin{equation}
\delta r=-\frac{c_{_{F}}}{8\beta_{_{\theta}}^{2}F''(\beta_{_{\theta}})\theta_{s}}\frac{\varepsilon S}{r_{_{0}}^{2}}.\label{eqt7}
\end{equation}

This estimate retains only leading-order terms in $\delta r/r_{_0}$ and $\delta\phi$ and requires $|\delta r|\ll r_{_0}$. It is not intended to predict complete film disappearance.

\begin{remark}\label{remdr}

Numerical evaluation of the prefactor in \eqref{eqt7} indicates that the predicted contraction increases with $\theta$ and tends to zero as $\theta\to0^{+}$. For identical geometric parameters, the logarithmic potential gives a smaller predicted contraction than the polynomial potential when $\theta<0.6948$, with the crossover value determined numerically. At fixed $r_{_0}$, $S$, and $\varepsilon$, the estimate is inversely proportional to $\theta_s$ for $0<\theta_s<\pi$. Reducing $\varepsilon$ or the domain area relative to $r_{_0}^{2}$ reduces the predicted shrinkage, provided the interface remains resolved and the assumptions of the estimate remain valid.
\end{remark}

\section{Upwind Lagrange multiplier approach}
\label{sec3}

The preceding analysis separates the structural laws of the model from its finite-width shrinkage. To retain those laws at the discrete level, we approximate \eqref{eqo11}--\eqref{eqo12} using bulk and wall Lagrange multipliers, upwind finite-volume fluxes, and ghost values for the boundary conditions. The multipliers enforce discrete energy identities, while the flux construction controls the phase-field bounds. The wall ghost relation couples these components and supplies the boundary dissipation term in the energy estimate.

First, applying the Lagrange multiplier approach to the governing equations~\eqref{eqo11}-\eqref{eqo12} yields
\begin{equation}
\left\{\begin{array}{l}
\frac{\partial\phi}{\partial t}=-\nabla\cdot\bm{J},\\[1mm]
\bm{J}=-M(\phi)\nabla\mu,\\[1mm]
\mu=-\varepsilon^2\Delta\phi+\xi(t)F'(\phi),
\end{array}\right.\quad \text{in}~\Omega\times(0,T],\label{eqn1}	
\end{equation}
subject to the following dynamic contact line boundary conditions
\begin{align}
\begin{cases}
\frac{\partial \phi}{\partial t}=-\kappa\left(\varepsilon^2 \frac{\partial \phi}{\partial \boldsymbol{n}}+\eta(t) g'(\phi)\right), \quad &\frac{\partial \mu}{\partial \boldsymbol{n}} = 0,\quad \text{on}~\Gamma_{w}, \\[1mm]
\varepsilon^2 \frac{\partial \phi}{\partial \boldsymbol{n}}=0, \quad 	&\frac{\partial \mu}{\partial \boldsymbol{n}} = 0,\quad \text{on} ~\Gamma_{n},
\end{cases} \label{eqn2}
\end{align}
with
\begin{equation}
\frac{\mathrm{d}}{\mathrm{d}t}\int_{\Omega}F(\phi) \mathrm{d}\bm{x}
=\xi(t)\int_{\Omega}F'(\phi)\frac{\partial\phi}{\partial t}\mathrm{d}\bm{x}\quad\text{and}\quad
\frac{\mathrm{d}}{\mathrm{d}t}\int_{\Gamma_{w}}g(\phi) \mathrm{d}s
=\eta(t)\int_{\Gamma_{w}}g'(\phi)\frac{\partial\phi}{\partial t} \mathrm{d}s,\label{eqn3}
\end{equation}

where $\xi(t)$ and $\eta(t)$ are scalar Lagrange multipliers. The choice $\xi(t)=\eta(t)=1$ recovers the original continuous system and satisfies \eqref{eqn3} by the chain rule. These identities motivate the discrete constraints below; they do not uniquely determine a multiplier when its associated energy derivative vanishes.

The multiplier identities address the energy balance; to control the phase-field bounds, we next construct an upwind flux. Define the positive and negative parts
\begin{equation}
\chi^{+}=\max\{\chi,0\},\quad\chi^{-}=\min\{\chi,0\},\label{eqn4}
\end{equation}

and, for $k=1$ in \eqref{eqo13}, define the two-state upwind mobility by
\begin{equation}
M(\chi_{1},\chi_{2})=(1+\chi_{1})^{+}(1-\chi_{2})^{+}.\label{eqn5}
\end{equation}

For equal states in $[-1,1]$, this definition satisfies $M(\phi,\phi)=1-\phi^2$. The positive parts specify the numerical flux and are not a post-processing cut-off of the phase variable.

Combining this mobility with the multiplier formulation gives the following finite-volume discretization of \eqref{eqn1}--\eqref{eqn3}.
Starting with the two-dimensional case,
we divide the computational domain $\overline{\Omega}\,(\triangleq\Omega\cup\partial\Omega)$ into $N_{x}\times N_{y}$ cells $C_{i,j}=[x_{i-\frac{1}{2}},x_{i+\frac{1}{2}}]\times[y_{j-\frac{1}{2}},y_{j+\frac{1}{2}}],\,i=1,2,\cdots,N_{x},\, j=1,2,\cdots,N_{y}$,
with spatial steps $\Delta x$ and $\Delta y$.
In each cell $C_{i,j}$, the corresponding cell average $\phi_{i,j}$ is defined as
\begin{equation}
\phi_{i,j}(t)=\frac{1}{\Delta x\Delta y}\iint_{C_{i,j}}\phi(x,y,t)\mathrm{d}x\mathrm{d}y.\label{eqn6}
\end{equation}
Applying the backward Euler method in time and the finite-volume method in space,
the continuous system~\eqref{eqn1}-\eqref{eqn2} is approximated as
\begin{align}
&\phi_{i,j}^{n+1}-\phi_{i,j}^{n}=-\left[\frac{\Delta t}{\Delta x}\left(J_{i+\frac{1}{2},j}^{n+1}-J_{i-\frac{1}{2},j}^{n+1}\right)+\frac{\Delta t}{\Delta y}\left(J_{i,j+\frac{1}{2}}^{n+1}-J_{i,j-\frac{1}{2}}^{n+1}\right)\right],\label{eqn7}\\[1mm]
&J_{i+\frac{1}{2},j}^{n+1}=\left(V_{i+\frac{1}{2},j}^{n+1}\right)^{+}M(\phi_{i,j}^{n+1},\phi_{i+1,j}^{n+1})
+\left(V_{i+\frac{1}{2},j}^{n+1}\right)^{-}M(\phi_{i+1,j}^{n+1},\phi_{i,j}^{n+1}),\label{eqn8}\\[1mm]
&V_{i+\frac{1}{2},j}^{n+1}=-\frac{1}{\Delta x}\left(\mu_{i+1,j}^{n+1}-\mu_{i,j}^{n+1}\right),	\label{eqn9}\\[1mm]
&J_{i,j+\frac{1}{2}}^{n+1}=\left(V_{i,j+\frac{1}{2}}^{n+1}\right)^{+}M(\phi_{i,j}^{n+1},\phi_{i,j+1}^{n+1})
+\left(V_{i,j+\frac{1}{2}}^{n+1}\right)^{-}M(\phi_{i,j+1}^{n+1},\phi_{i,j}^{n+1}),\label{eqn10}\\[1mm]
&V_{i,j+\frac{1}{2}}^{n+1}=-\frac{1}{\Delta y}\left(\mu_{i,j+1}^{n+1}-\mu_{i,j}^{n+1}\right),\label{eqn11}\\[1mm]
&\mu_{i,j}^{n+1}=-\varepsilon^2(\Delta\phi)_{i,j}^{n+1}+\xi^{n+1}F'(\phi_{i,j}^{n+1}),
\label{eqn12}
\end{align}
subject to the following dynamic contact line boundary conditions
\begin{align}
\begin{cases}
	\phi_{i,1}^{n+1}-\phi_{i,1}^{n}=-\kappa\Delta t\left(-\varepsilon^2 \left(\partial_{y}\phi\right)_{i,\frac{1}{2}}^{n+1}+\eta^{n+1}g'(\phi_{i,1}^{n+1})\right), \quad & \text{on}~\Gamma_{w}, \\[1mm]
	\left(\partial_{x}\phi\right)_{\frac{1}{2},j}^{n+1}=0, \quad\left(\partial_{x}\phi\right)_{N_{x}+\frac{1}{2},j}^{n+1}=0, \quad  \left(\partial_{y}\phi\right)_{i,N_{y}+\frac{1}{2}}^{n+1}=0, \quad 	& \text{on}~\Gamma_{n},
\end{cases} \label{eqn15}
\end{align}
and no-flux boundary conditions implemented by
\begin{equation}
J_{i,\frac{1}{2}}^{n+1} = 0,\quad
J_{\frac{1}{2},j}^{n+1} = 0,\quad J_{N_{x}+\frac{1}{2},j}^{n+1} = 0,\quad	J_{i,N_{y}+\frac{1}{2}}^{n+1} = 0.\label{eqn16}
\end{equation}
The energy identities in \eqref{eqn3} are discretized as follows:
\begin{align} &\sum_{i=1}^{N_{x}}\sum_{j=1}^{N_{y}}\left(F(\phi_{i,j}^{n+1})-F(\phi_{i,j}^{n})\right)=\xi^{n+1}\sum_{i=1}^{N_{x}}\sum_{j=1}^{N_{y}}F'(\phi_{i,j}^{n+1})(\phi_{i,j}^{n+1}-\phi_{i,j}^{n}),\label{eqn13}\\	&\sum_{i=1}^{N_{x}}\left(g(\phi_{i,1}^{n+1})-g(\phi_{i,1}^{n})\right)=\eta^{n+1}\sum_{i=1}^{N_{x}}g'(\phi_{i,1}^{n+1})\left(\phi_{i,1}^{n+1}-\phi_{i,1}^{n}\right).\label{eqn14}
\end{align}

Here $\Delta t>0$ is the time step and $t^n=n\Delta t$, with $n=0,\ldots,N$ and $T=N\Delta t$. Superscripts denote time levels; cell and face indices specify the locations of the phase variable, chemical potential, and fluxes. The multipliers $\xi^{n+1}$ and $\eta^{n+1}$ are spatially uniform scalars.

\begingroup\emergencystretch=1em

To impose the boundary conditions, ghost values $\phi_{0,j}^{n+1}$, $\phi_{N_{x}+1,j}^{n+1}$, $\phi_{i,0}^{n+1}$, $\phi_{i,N_{y}+1}^{n+1}$
($i=1,2,\cdots,N_{x},\, j=1,2,\cdots,N_{y}$) are introduced outside the boundary $\partial\Omega$.
Using the central difference scheme to discretize the spatial derivatives in \eqref{eqn15}, the following relations are obtained:
\par\endgroup
\begin{align}
\begin{cases}
\phi_{0,j}^{n+1}=\phi_{1,j}^{n+1},& \phi_{N_{x}+1,j}^{n+1}=\phi_{N_{x},j}^{n+1},\quad j=1,2,\cdots,N_{y},\\[1.5mm]
\phi_{i,0}^{n+1}=\left(1-\frac{\Delta y}{\varepsilon^{2}\kappa\Delta t}\right)\phi_{i,1}^{n+1}
  +\frac{\Delta y}{\varepsilon^{2}\kappa\Delta t}\phi_{i,1}^{n}-\frac{\Delta y}{\varepsilon^{2}}\eta^{n+1}g'(\phi_{i,1}^{n+1}),&
\phi_{i,N_{y}+1}^{n+1}=\phi_{i,N_{y}}^{n+1},\quad ~i=1,2,\cdots,N_{x}.	
\end{cases}\label{eqghost}
\end{align}
Therefore, for all $i=1,2,\cdots,N_{x}$ and $j=1,2,\cdots,N_{y}$,
the first partial derivatives and Laplacian can be discretized as
\begin{equation} (\partial_{x}\phi)_{i+\frac{1}{2},j}^{n+1}=\frac{\phi_{i+1,j}^{n+1}-\phi_{i,j}^{n+1}}{\Delta x}\quad\text{and}\quad
(\partial_{y}\phi)_{i,j+\frac{1}{2}}^{n+1}=\frac{\phi_{i,j+1}^{n+1}-\phi_{i,j}^{n+1}}{\Delta y},\label{eqn17}
\end{equation}
and
\begin{align}
(\Delta\phi)_{i,j}^{n+1}=(\partial_{x}^{2}\phi)_{i,j}^{n+1}+(\partial_{y}^{2}\phi)_{i,j}^{n+1}
=\frac{\phi_{i+1,j}^{n+1}-2\phi_{i,j}^{n+1}+\phi_{i-1,j}^{n+1}}{\Delta x^{2}}
  +\frac{\phi_{i,j+1}^{n+1}-2\phi_{i,j}^{n+1}+\phi_{i,j-1}^{n+1}}{\Delta y^{2}}\label{eqn19},
\end{align}
respectively.

With the fluxes and boundary treatment specified, we can now verify that their coupling retains the desired structure. The following results establish boundedness, mass conservation, and energy dissipation for solutions of \eqref{eqn7}--\eqref{eqn19}. These are structural statements about the discrete equations; they do not establish existence or uniqueness of the nonlinear update, or convergence of a particular nonlinear solver.

\begin{thm}\label{thm1} (Boundedness)
The fully discrete scheme \eqref{eqn7}-\eqref{eqn19} 
preserves strict bounds on the cell averages $\phi_{i,j}$.
That is, for all $i,j$, if $|\phi_{i,j}^{n}|<1$, then $|\phi_{i,j}^{n+1}|<1$.
\end{thm}
		
\begin{proof}
We first prove that $|\phi_{i,j}^{n}|<1$ implies $\phi_{i,j}^{n+1}<1$ for all $i,j$.
Suppose, to the contrary, that there exists a group of contiguous cells $\{\phi_{i,j}^{n+1}:\alpha\le i\le k,\ \beta\le j\le l\}$ such that $\phi_{i,j}^{n+1}\ge1$. Note that the proof remains valid if the group has only one point, i.e., $\alpha=i=k$, $\beta=j=l$.
Then, summing both sides of \eqref{eqn7} over these cells yields
\begin{align}
0<&\,\frac{1}{\Delta t}\sum_{i=\alpha}^{k}\sum_{j=\beta}^{l}(\phi_{i,j}^{n+1}-\phi_{i,j}^{n})
=-\frac{1}{\Delta x}\sum_{i=\alpha}^{k}\sum_{j=\beta}^{l}\left(J_{i+\frac{1}{2},j}^{n+1}-J_{i-\frac{1}{2},j}^{n+1}\right)-\frac{1}{\Delta y}\sum_{i=\alpha}^{k}\sum_{j=\beta}^{l}\left(J_{i,j+\frac{1}{2}}^{n+1}-J_{i,j-\frac{1}{2}}^{n+1}\right)\nonumber\\
=&\,\frac{1}{\Delta x}\sum_{j=\beta}^{l}\left(J_{\alpha-\frac{1}{2},j}^{n+1}-J_{k+\frac{1}{2},j}^{n+1}\right)+\frac{1}{\Delta y}\sum_{i=\alpha}^{k}\left(J_{i,\beta-\frac{1}{2}}^{n+1}-J_{i,l+\frac{1}{2}}^{n+1}\right)\nonumber\\
=&\,\frac{1}{\Delta x}\sum_{j=\beta}^{l}\left(\left(V_{\alpha-\frac{1}{2},j}^{n+1}\right)^{+}M(\phi_{\alpha-1,j}^{n+1},\phi_{\alpha,j}^{n+1})
+\left(V_{\alpha-\frac{1}{2},j}^{n+1}\right)^{-}M(\phi_{\alpha,j}^{n+1},\phi_{\alpha-1,j}^{n+1})\right)\nonumber\\
&-\frac{1}{\Delta x}\sum_{j=\beta}^{l}\left(\left(V_{k+\frac{1}{2},j}^{n+1}\right)^{+}M(\phi_{k,j}^{n+1},\phi_{k+1,j}^{n+1})
+\left(V_{k+\frac{1}{2},j}^{n+1}\right)^{-}M(\phi_{k+1,j}^{n+1},\phi_{k,j}^{n+1})\right)\nonumber\\
&+\frac{1}{\Delta y}\sum_{i=\alpha}^{k}\left(\left(V_{i,\beta-\frac{1}{2}}^{n+1}\right)^{+}M(\phi_{i,\beta-1}^{n+1},\phi_{i,\beta}^{n+1})
+\left(V_{i,\beta-\frac{1}{2}}^{n+1}\right)^{-}M(\phi_{i,\beta}^{n+1},\phi_{i,\beta-1}^{n+1})\right)\nonumber\\
&-\frac{1}{\Delta y}\sum_{i=\alpha}^{k}\left(\left(V_{i,l+\frac{1}{2}}^{n+1}\right)^{+}M(\phi_{i,l}^{n+1},\phi_{i,l+1}^{n+1})
+\left(V_{i,l+\frac{1}{2}}^{n+1}\right)^{-}M(\phi_{i,l+1}^{n+1},\phi_{i,l}^{n+1})\right).\label{tha1}
\end{align}

By the definition \eqref{eqn5} and the inequalities $\phi_{i,\beta}^{n+1}\geq 1,\,\phi_{i,l}^{n+1}\geq 1,\,\phi_{\alpha,j}^{n+1}\geq 1 $ and $\phi_{k,j}^{n+1}\geq 1,$ we have
\begin{align}
&M(\phi_{\alpha-1,j}^{n+1},\phi_{\alpha,j}^{n+1})=0,\quad M(\phi_{\alpha,j}^{n+1},\phi_{\alpha-1,j}^{n+1})\geq0,\quad
M(\phi_{k,j}^{n+1},\phi_{k+1,j}^{n+1})\geq0,\quad M(\phi_{k+1,j}^{n+1},\phi_{k,j}^{n+1})=0,\nonumber\\[1mm]
&M(\phi_{i,\beta-1}^{n+1},\phi_{i,\beta}^{n+1})=0,\quad M(\phi_{i,\beta}^{n+1},\phi_{i,\beta-1}^{n+1})\geq0,\quad
M(\phi_{i,l}^{n+1},\phi_{i,l+1}^{n+1})\geq0,\quad M(\phi_{i,l+1}^{n+1},\phi_{i,l}^{n+1})=0,
\label{tha2}
\end{align}
where $i=\alpha,\alpha+1,\alpha+2,\cdots,k$ and $ j=\beta,\beta+1,\beta+2,\cdots,l.$
Therefore, the right-hand side of \eqref{tha1} must be non-positive, which contradicts the strict positivity of the left-hand side. Hence $\phi_{i,j}^{n+1}<1$.

The analogous argument at the lower bound gives $\phi_{i,j}^{n+1}>-1$.
\end{proof}


\begin{thm}\label{thm2} (Mass conservation)
The fully discrete scheme \eqref{eqn7}-\eqref{eqn19} ensures that the total mass is conserved during the evolution, i.e.,
\begin{equation}
m^{n+1}\triangleq\sum_{i=1}^{N_{x}}\sum_{j=1}^{N_{y}}\phi_{i,j}^{n+1}=\sum_{i=1}^{N_{x}}\sum_{j=1}^{N_{y}}\phi_{i,j}^{n}=\cdots
=\sum_{i=1}^{N_{x}}\sum_{j=1}^{N_{y}}\phi_{i,j}^{0}.\label{thb1}
\end{equation}
\end{thm}

\begin{proof}
Summing both sides of \eqref{eqn7} over all cells $C_{i,j}$ yields
\begin{align}
\sum_{i=1}^{N_{x}}\sum_{j=1}^{N_{y}}(\phi_{i,j}^{n+1}-\phi_{i,j}^{n})
=&\,-\frac{\Delta t}{\Delta x}\sum_{i=1}^{N_{x}}\sum_{j=1}^{N_{y}}\left(J_{i+\frac{1}{2},j}^{n+1}-J_{i-\frac{1}{2},j}^{n+1}\right)-\frac{\Delta t}{\Delta y}\sum_{i=1}^{N_{x}}\sum_{j=1}^{N_{y}}\left(J_{i,j+\frac{1}{2}}^{n+1}-J_{i,j-\frac{1}{2}}^{n+1}\right)\nonumber\\
=&\,-\frac{\Delta t}{\Delta x}\sum_{j=1}^{N_{y}}\left(J_{{N_{x}+\frac{1}{2},j}}^{n+1}-J_{\frac{1}{2},j}^{n+1}\right)-\frac{\Delta t}{\Delta y}\sum_{i=1}^{N_{x}}\left(J_{i,{N_{y}+\frac{1}{2}}}^{n+1}-J_{i,\frac{1}{2}}^{n+1}\right)=0,\label{thb2}
\end{align}

where the last equality follows from the no-flux conditions \eqref{eqn16}. The physical discrete mass includes the constant cell-area factor $\Delta x\Delta y$, which is omitted from $m^n$.
\end{proof}

\begin{thm}\label{thm3} (Energy dissipation)
The fully discrete scheme \eqref{eqn7}-\eqref{eqn19} is energy stable and satisfies the following discrete energy dissipation law:
\begin{align}
	\frac{\mathcal{W}^{n+1}-\mathcal{W}^{n}}{\Delta t}
	\le&\, -\Delta x\Delta y\sum_{i=1}^{N_{x}-1}\sum_{j=1}^{N_{y}}
	\min\left\{M(\phi_{i,j}^{n+1},\phi_{i+1,j}^{n+1}),M(\phi_{i+1,j}^{n+1},\phi_{i,j}^{n+1})\right\}
	\left|V_{i+\frac{1}{2},j}^{n+1}\right|^{2}\nonumber\\
	&-\Delta x\Delta y\sum_{i=1}^{N_{x}}\sum_{j=1}^{N_{y}-1}
	\min\left\{M(\phi_{i,j}^{n+1},\phi_{i,j+1}^{n+1}),M(\phi_{i,j+1}^{n+1},\phi_{i,j}^{n+1})\right\}
	\left|V_{i,j+\frac{1}{2}}^{n+1}\right|^{2}
	\le 0,\label{thc1}
\end{align}
where
\begin{align}
\mathcal{W}^{n}
=&\,\Delta x\Delta y\sum_{i=1}^{N_{x}-1}\sum_{j=1}^{N_{y}}\frac{\varepsilon^{2}}{2}\left(\frac{\phi_{i+1,j}^{n}-\phi_{i,j}^{n}}{\Delta x}\right)^{2}
  +\Delta x\Delta y\sum_{i=1}^{N_{x}}\sum_{j=1}^{N_{y}-1}\frac{\varepsilon^{2}}{2}\left(\frac{\phi_{i,j+1}^{n}-\phi_{i,j}^{n}}{\Delta y}\right)^{2}
\nonumber\\
&+\Delta x\Delta y\sum_{i=1}^{N_{x}}\sum_{j=1}^{N_{y}}F(\phi_{i,j}^{n})+\Delta x\sum_{i=1}^{N_{x}}g(\phi_{i,1}^{n}).\label{thc2}
\end{align}
\end{thm}

\begin{proof}
Subtracting the first term on the right-hand side of the discrete energy \eqref{thc2} at consecutive time levels
and using the ghost point relation \eqref{eqghost} together with the identity $a^{2}-b^{2}=2a(a-b)-(a-b)^{2}$, we obtain
\begin{align}
&\,\frac{\varepsilon^{2}}{2}\sum_{i=1}^{N_{x}-1}\sum_{j=1}^{N_{y}}\left[\left(\frac{\phi_{i+1,j}^{n+1}-\phi_{i,j}^{n+1}}{\Delta x}\right)^{2}
				-\left(\frac{\phi_{i+1,j}^{n}-\phi_{i,j}^{n}}{\Delta x}\right)^{2}\right]\nonumber\\
=&\,\varepsilon^{2}\sum_{i=1}^{N_{x}-1}\sum_{j=1}^{N_{y}}\left(\frac{\phi_{i+1,j}^{n+1}-\phi_{i,j}^{n+1}}{\Delta x^{2}}\right)\cdot
				\left[(\phi_{i+1,j}^{n+1}-\phi_{i+1,j}^{n})-(\phi_{i,j}^{n+1}-\phi_{i,j}^{n})\right]\nonumber\\
 &\,-\frac{\varepsilon^{2}}{2}\sum_{i=1}^{N_{x}-1}\sum_{j=1}^{N_{y}}\left[\left(\frac{\phi_{i+1,j}^{n+1}-\phi_{i,j}^{n+1}}{\Delta x}\right)
				-\left(\frac{\phi_{i+1,j}^{n}-\phi_{i,j}^{n}}{\Delta x}\right)\right]^{2}\nonumber\\
=&\,\varepsilon^{2}\sum_{i=2}^{N_{x}}\sum_{j=1}^{N_{y}}\left(\frac{\phi_{i,j}^{n+1}-\phi_{i-1,j}^{n+1}}{\Delta x^{2}}\right)
				\cdot(\phi_{i,j}^{n+1}-\phi_{i,j}^{n})
	-\varepsilon^{2}\sum_{i=1}^{N_{x}-1}\sum_{j=1}^{N_{y}}\left(\frac{\phi_{i+1,j}^{n+1}-\phi_{i,j}^{n+1}}{\Delta x^{2}}\right)
				\cdot(\phi_{i,j}^{n+1}-\phi_{i,j}^{n})\nonumber\\
&\,-\frac{\varepsilon^{2}}{2}\sum_{i=1}^{N_{x}-1}\sum_{j=1}^{N_{y}}\left[\left(\frac{\phi_{i+1,j}^{n+1}-\phi_{i,j}^{n+1}}{\Delta x}\right)
				-\left(\frac{\phi_{i+1,j}^{n}-\phi_{i,j}^{n}}{\Delta x}\right)\right]^{2}\nonumber\\
=&\,-\varepsilon^{2}\sum_{i=2}^{N_{x}-1}\sum_{j=1}^{N_{y}}
	\left(\frac{\phi_{i+1,j}^{n+1}-2\phi_{i,j}^{n+1}+\phi_{i-1,j}^{n+1}}{\Delta x^{2}}\right)\cdot(\phi_{i,j}^{n+1}-\phi_{i,j}^{n})\nonumber\\
&\,+\varepsilon^{2}\sum_{j=1}^{N_{y}}\left(\frac{\phi_{N_{x},j}^{n+1}-\phi_{N_{x}-1,j}^{n+1}}{\Delta x^{2}}\right)
				\cdot(\phi_{N_{x},j}^{n+1}-\phi_{N_{x},j}^{n})
	-\varepsilon^{2}\sum_{j=1}^{N_{y}}\left(\frac{\phi_{2,j}^{n+1}-\phi_{1,j}^{n+1}}{\Delta x^{2}}\right)
				\cdot(\phi_{1,j}^{n+1}-\phi_{1,j}^{n})\nonumber\\
&\,-\frac{\varepsilon^{2}}{2}\sum_{i=1}^{N_{x}-1}\sum_{j=1}^{N_{y}}\left[\left(\frac{\phi_{i+1,j}^{n+1}-\phi_{i,j}^{n+1}}{\Delta x}\right)
				-\left(\frac{\phi_{i+1,j}^{n}-\phi_{i,j}^{n}}{\Delta x}\right)\right]^{2}\nonumber\\
=&\,-\varepsilon^{2}\sum_{i=1}^{N_{x}}\sum_{j=1}^{N_{y}}(\partial_{x}^{2}\phi)_{i,j}^{n+1}\cdot(\phi_{i,j}^{n+1}-\phi_{i,j}^{n})
   -\frac{\varepsilon^{2}}{2}\sum_{i=1}^{N_{x}-1}\sum_{j=1}^{N_{y}}\left[\left(\frac{\phi_{i+1,j}^{n+1}-\phi_{i,j}^{n+1}}{\Delta x}\right)
				-\left(\frac{\phi_{i+1,j}^{n}-\phi_{i,j}^{n}}{\Delta x}\right)\right]^{2}.\label{thc3}
\end{align}
Applying a similar treatment to the second term on the right-hand side of the discrete energy \eqref{thc2} gives:
\begin{align}
&\,\frac{\varepsilon^{2}}{2}\sum_{i=1}^{N_{x}}\sum_{j=1}^{N_{y}-1}\left[\left(\frac{\phi_{i,j+1}^{n+1}-\phi_{i,j}^{n+1}}{\Delta y}\right)^{2}
	-\left(\frac{\phi_{i,j+1}^{n}-\phi_{i,j}^{n}}{\Delta y}\right)^{2}\right]\nonumber\\
=&\,-\varepsilon^{2}\sum_{i=1}^{N_{x}}\sum_{j=2}^{N_{y}-1}
	\left(\frac{\phi_{i,j+1}^{n+1}-2\phi_{i,j}^{n+1}+\phi_{i,j-1}^{n+1}}{\Delta y^{2}}\right)\cdot(\phi_{i,j}^{n+1}-\phi_{i,j}^{n})\nonumber\\
&+\varepsilon^{2}\sum_{i=1}^{N_{x}}\left(\frac{\phi_{i,N_{y}}^{n+1}-\phi_{i,N_{y}-1}^{n+1}}{\Delta y^{2}}\right)
			\cdot(\phi_{i,N_{y}}^{n+1}-\phi_{i,N_{y}}^{n})
-\varepsilon^{2}\sum_{i=1}^{N_{x}}\left(\frac{\phi_{i,2}^{n+1}-\phi_{i,1}^{n+1}}{\Delta y^{2}}\right)
        \cdot(\phi_{i,1}^{n+1}-\phi_{i,1}^{n})\nonumber\\
&-\frac{\varepsilon^{2}}{2}\sum_{i=1}^{N_{x}}\sum_{j=1}^{N_{y}-1}\left[\left(\frac{\phi_{i,j+1}^{n+1}-\phi_{i,j}^{n+1}}{\Delta y}\right)
			-\left(\frac{\phi_{i,j+1}^{n}-\phi_{i,j}^{n}}{\Delta y}\right)\right]^{2}\nonumber\\
=&\,-\varepsilon^{2}\sum_{i=1}^{N_{x}}\sum_{j=1}^{N_{y}}(\partial_{y}^{2}\phi)_{i,j}^{n+1}\cdot(\phi_{i,j}^{n+1}-\phi_{i,j}^{n})
   -\varepsilon^{2}\sum_{i=1}^{N_{x}}\left(\frac{\phi_{i,1}^{n+1}-\phi_{i,0}^{n+1}}{\Delta y^{2}}\right)
    \cdot\left(\phi_{i,1}^{n+1}-\phi_{i,1}^{n}\right)\nonumber\\
 &\,-\frac{\varepsilon^{2}}{2}\sum_{i=1}^{N_{x}}\sum_{j=1}^{N_{y}-1}\left[\left(\frac{\phi_{i,j+1}^{n+1}-\phi_{i,j}^{n+1}}{\Delta y}\right)
				-\left(\frac{\phi_{i,j+1}^{n}-\phi_{i,j}^{n}}{\Delta y}\right)\right]^{2}.\label{thc4}
\end{align}

Multiplying both sides of \eqref{eqn7} by $\mu_{i,j}^{n+1}$ and summing over all cells,
and using the no-flux boundary conditions \eqref{eqn16}, yields
\begin{align}			
&\sum_{i=1}^{N_{x}}\sum_{j=1}^{N_{y}}\left(\phi_{i,j}^{n+1}-\phi_{i,j}^{n}\right)\mu_{i,j}^{n+1}
=-\sum_{i=1}^{N_{x}}\sum_{j=1}^{N_{y}}\left[\frac{\Delta t}{\Delta x}\left(J_{i+\frac{1}{2},j}^{n+1}-J_{i-\frac{1}{2},j}^{n+1}\right)
   +\frac{\Delta t}{\Delta y}\left(J_{i,j+\frac{1}{2}}^{n+1}-J_{i,j-\frac{1}{2}}^{n+1}\right)\right]\mu_{i,j}^{n+1}\nonumber\\
=&\,\sum_{i=1}^{N_{x}-1}\sum_{j=1}^{N_{y}}\frac{\Delta t}{\Delta x}J_{i+\frac{1}{2},j}^{n+1}\left(\mu_{i+1,j}^{n+1}-\mu_{i,j}^{n+1}\right)
   +\sum_{i=1}^{N_{x}}\sum_{j=1}^{N_{y}-1}\frac{\Delta t}{\Delta y}J_{i,j+\frac{1}{2}}^{n+1}\left(\mu_{i,j+1}^{n+1}-\mu_{i,j}^{n+1}\right)\nonumber\\
=&\,-\Delta t\sum_{i=1}^{N_{x}-1}\sum_{j=1}^{N_{y}}J_{i+\frac{1}{2},j}^{n+1}V_{i+\frac{1}{2},j}^{n+1}
    -\Delta t\sum_{i=1}^{N_{x}}\sum_{j=1}^{N_{y}-1}J_{i,j+\frac{1}{2}}^{n+1}V_{i,j+\frac{1}{2}}^{n+1}\nonumber\\
=&\,-\Delta t\sum_{i=1}^{N_{x}-1}\sum_{j=1}^{N_{y}}\left[\left(V_{i+\frac{1}{2},j}^{n+1}\right)^{+}M(\phi_{i,j}^{n+1},\phi_{i+1,j}^{n+1})
   +\left(V_{i+\frac{1}{2},j}^{n+1}\right)^{-}M(\phi_{i+1,j}^{n+1},\phi_{i,j}^{n+1})\right]V_{i+\frac{1}{2},j}^{n+1}\nonumber\\
&-\Delta t\sum_{i=1}^{N_{x}}\sum_{j=1}^{N_{y}-1}\left[\left(V_{i,j+\frac{1}{2}}^{n+1}\right)^{+}M(\phi_{i,j}^{n+1},\phi_{i,j+1}^{n+1})
   +\left(V_{i,j+\frac{1}{2}}^{n+1}\right)^{-}M(\phi_{i,j+1}^{n+1},\phi_{i,j}^{n+1})\right]V_{i,j+\frac{1}{2}}^{n+1}\nonumber\\
\le&\, -\Delta t\sum_{i=1}^{N_{x}-1}\sum_{j=1}^{N_{y}}
      \min\left\{M(\phi_{i,j}^{n+1},\phi_{i+1,j}^{n+1}),M(\phi_{i+1,j}^{n+1},\phi_{i,j}^{n+1})\right\}
	\left|V_{i+\frac{1}{2},j}^{n+1}\right|^{2}\nonumber\\
&-\Delta t\sum_{i=1}^{N_{x}}\sum_{j=1}^{N_{y}-1}
			\min\left\{M(\phi_{i,j}^{n+1},\phi_{i,j+1}^{n+1}),M(\phi_{i,j+1}^{n+1},\phi_{i,j}^{n+1})\right\}
	\left|V_{i,j+\frac{1}{2}}^{n+1}\right|^{2}
\le0.\label{thc5}
\end{align}

Finally, by combining~\eqref{thc3}-\eqref{thc5} and using \eqref{eqn19}, we have
\begin{align}
\frac{\mathcal{W}^{n+1}-\mathcal{W}^{n}}{\Delta x\Delta y}
=&\,\frac{\varepsilon^{2}}{2}\sum_{i=1}^{N_{x}-1}\sum_{j=1}^{N_{y}}\left[\left(\frac{\phi_{i+1,j}^{n+1}-\phi_{i,j}^{n+1}}{\Delta x}\right)^{2}
			-\left(\frac{\phi_{i+1,j}^{n}-\phi_{i,j}^{n}}{\Delta x}\right)^{2}\right]\nonumber\\
 &\,+\frac{\varepsilon^{2}}{2}\sum_{i=1}^{N_{x}}\sum_{j=1}^{N_{y}-1}\left[\left(\frac{\phi_{i,j+1}^{n+1}-\phi_{i,j}^{n+1}}{\Delta y}\right)^{2}
	-\left(\frac{\phi_{i,j+1}^{n}-\phi_{i,j}^{n}}{\Delta y}\right)^{2}\right]\nonumber\\
 &\,+\sum_{i=1}^{N_{x}}\sum_{j=1}^{N_{y}}\left[F(\phi_{i,j}^{n+1})-F(\phi_{i,j}^{n})\right]
    +\frac{1}{\Delta y}\sum_{i=1}^{N_{x}}\left[g(\phi_{i,1}^{n+1})-g(\phi_{i,1}^{n})\right]\nonumber\\
\le&\,-\varepsilon^{2}\sum_{i=1}^{N_{x}}\sum_{j=1}^{N_{y}}\left[(\partial_{x}^{2}\phi)_{i,j}^{n+1}+(\partial_{y}^{2}\phi)_{i,j}^{n+1}\right](\phi_{i,j}^{n+1}-\phi_{i,j}^{n})\nonumber\\
&\,-\varepsilon^{2}\sum_{i=1}^{N_{x}}\left(\frac{\phi_{i,1}^{n+1}-\phi_{i,0}^{n+1}}{\Delta y^{2}}\right)\left(\phi_{i,1}^{n+1}-\phi_{i,1}^{n}\right)\nonumber\\
 &\,+\xi^{n+1}\sum_{i=1}^{N_{x}}\sum_{j=1}^{N_{y}}F'(\phi_{i,j}^{n+1})(\phi_{i,j}^{n+1}-\phi_{i,j}^{n})
    +\frac{1}{\Delta y}\eta^{n+1}\sum_{i=1}^{N_{x}}g'\left(\phi_{i,1}^{n+1}\right)\left(\phi_{i,1}^{n+1}-\phi_{i,1}^{n}\right)\nonumber\\
\le&\,-\varepsilon^{2}\sum_{i=1}^{N_{x}}\sum_{j=1}^{N_{y}}(\Delta\phi)_{i,j}^{n+1}(\phi_{i,j}^{n+1}-\phi_{i,j}^{n})-\frac{1}{\kappa\Delta t\Delta y}\sum_{i=1}^{N_{x}}\left(\phi_{i,1}^{n+1}-\phi_{i,1}^{n}\right)^{2}\nonumber\\
   &\,+\xi^{n+1}\sum_{i=1}^{N_{x}}\sum_{j=1}^{N_{y}}F'(\phi_{i,j}^{n+1})(\phi_{i,j}^{n+1}-\phi_{i,j}^{n})\nonumber\\
\le&\,-\varepsilon^{2}\sum_{i=1}^{N_{x}}\sum_{j=1}^{N_{y}}(\Delta\phi)_{i,j}^{n+1}(\phi_{i,j}^{n+1}-\phi_{i,j}^{n})
    +{\xi}^{n+1}\sum_{i=1}^{N_{x}}\sum_{j=1}^{N_{y}}F'(\phi_{i,j}^{n+1})(\phi_{i,j}^{n+1}-\phi_{i,j}^{n})\nonumber\\
=&\,\sum_{i=1}^{N_{x}}\sum_{j=1}^{N_{y}}(\phi_{i,j}^{n+1}-\phi_{i,j}^{n})\mu_{i,j}^{n+1}\nonumber\\
\le&\, -\Delta t\sum_{i=1}^{N_{x}-1}\sum_{j=1}^{N_{y}} \min\left\{M(\phi_{i,j}^{n+1},\phi_{i+1,j}^{n+1}),M(\phi_{i+1,j}^{n+1},\phi_{i,j}^{n+1})\right\}\left|V_{i+\frac{1}{2},j}^{n+1}\right|^{2}\nonumber\\
&-\Delta t\sum_{i=1}^{N_{x}}\sum_{j=1}^{N_{y}-1}\min\left\{M(\phi_{i,j}^{n+1},\phi_{i,j+1}^{n+1}),M(\phi_{i,j+1}^{n+1},\phi_{i,j}^{n+1})\right\}
	\left|V_{i,j+\frac{1}{2}}^{n+1}\right|^{2}\le 0.\label{thc6}
\end{align}	
\end{proof}

\section{Dimensional-splitting technique}
\label{sec4}

The fully coupled scheme provides the desired structural guarantees, but its nonlinear update involves the entire spatial grid. To reduce the size of each solve while retaining those guarantees, we update one row or column at a time. Each update uses the full discrete energy, including transverse gradient contributions, while changing only the cells in the active row or column. The resulting sequential splitting preserves the bounds, mass, and energy dissipation established in Section~\ref{sec3}.

\begin{figure}[!htp]
\centering
\includegraphics[width=8cm]{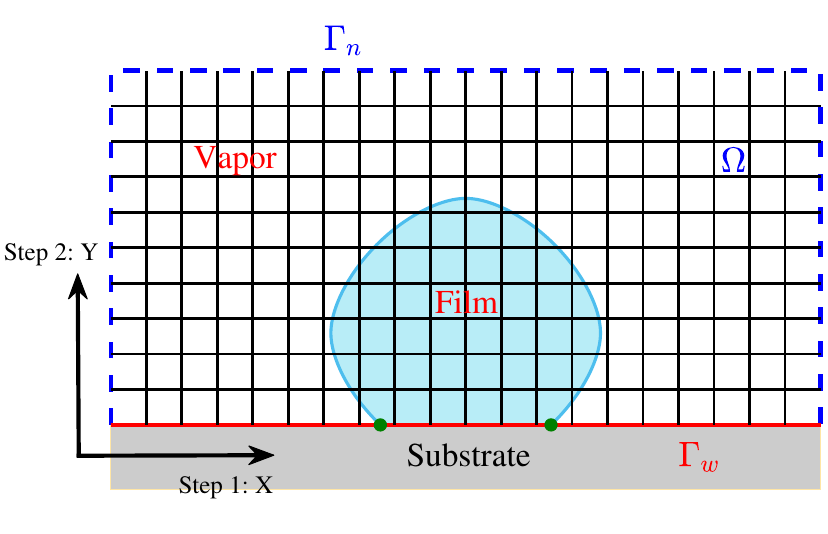}
\caption{
Sequential row and column updates in the dimensional-splitting scheme.}\label{figdsm}
\end{figure}

Specifically, at each time step $n$, we first sweep along the $x$-direction:
for each fixed row index $q=1,2,\cdots,N_{y}$, the cells on that row are updated by solving a one-dimensional Cahn--Hilliard system.
Let $\widetilde{\phi}_{i,j}^{n,q}$ denote the intermediate solution after the $q$-th row has been processed,
with the initialization $\widetilde{\phi}_{i,j}^{n,0}=\phi_{i,j}^{n}$.
After all rows are swept, we then sweep along the $y$-direction:
for each fixed column index $p=1,2,\cdots,N_{x}$, the cells on that column are updated analogously,
yielding $\widehat{\phi}_{i,j}^{n,p}$ with $\widehat{\phi}_{i,j}^{n,0}=\widetilde{\phi}_{i,j}^{n,N_{y}}$.
This alternating-direction procedure is illustrated in Fig.~\ref{figdsm}, and the full scheme is formulated as follows.

$\mathbf{Step~1.~for~q=1,2,\cdots,N_{y}~do:}$
\begin{align}
&\widetilde{\phi}_{i,j}^{n,q}-\widetilde{\phi}_{i,j}^{n,q-1}=
\left\{\begin{array}{ll}
-\frac{\Delta t}{\Delta x}\left(\widetilde{J}_{i+\frac{1}{2},j}^{n,q}-\widetilde{J}_{i-\frac{1}{2},j}^{n,q}\right),& \text{if}~j=q;\\[1mm]
0,& \text{otherwise},
\end{array}\right.\label{eqs1}\\[1mm]
&\widetilde{J}_{i+\frac{1}{2},j}^{n,q}=\left(\widetilde{V}_{i+\frac{1}{2},j}^{n,q}\right)^{+}M(\widetilde{\phi}_{i,j}^{n,q},\widetilde{\phi}_{i+1,j}^{n,q})
+\left(\widetilde{V}_{i+\frac{1}{2},j}^{n,q}\right)^{-}M(\widetilde{\phi}_{i+1,j}^{n,q},\widetilde{\phi}_{i,j}^{n,q}),\label{eqs2}\\[1mm]
&\widetilde{V}_{i+\frac{1}{2},j}^{n,q}=-\frac{1}{\Delta x}\left(\widetilde{\mu}_{i+1,j}^{n,q}-\widetilde{\mu}_{i,j}^{n,q}\right),\label{eqs3}\\[2mm]
&\widetilde{\mu}_{i,j}^{n,q}=-\varepsilon^{2}(\Delta\widetilde{\phi})_{i,j}^{n,q}+\widetilde{\xi}^{n,q}F'(\widetilde{\phi}_{i,j}^{n,q}),\label{eqs4}\\[1mm]
&\sum_{i=1}^{N_{x}}\sum_{j=1}^{N_{y}}\left(F(\widetilde{\phi}_{i,j}^{n,q})-F(\widetilde{\phi}_{i,j}^{n,q-1})\right)=\widetilde{\xi}^{n,q}\sum_{i=1}^{N_{x}}\sum_{j=1}^{N_{y}}F'(\widetilde{\phi}_{i,j}^{n,q})(\widetilde{\phi}_{i,j}^{n,q}-\widetilde{\phi}_{i,j}^{n,q-1}),\label{eqs5}\\[1mm]
&\sum_{i=1}^{N_{x}}\left(g(\widetilde{\phi}_{i,1}^{n,q})-g(\widetilde{\phi}_{i,1}^{n,q-1})\right)=\widetilde{\eta}^{n,q}\sum_{i=1}^{N_{x}}g'(\widetilde{\phi}_{i,1}^{n,q})\left(\widetilde{\phi}_{i,1}^{n,q}-\widetilde{\phi}_{i,1}^{n,q-1}\right),\label{eqs6}
\end{align}
subject to the following dynamic contact line boundary conditions
\begin{equation}
\frac{\widetilde{\phi}_{i,1}^{n,q}-\widetilde{\phi}_{i,1}^{n,q-1}}{\kappa\Delta t}=
\begin{cases}
\varepsilon^{2} (\partial_{y}\widetilde{\phi})_{i,\frac{1}{2}}^{n,q}-\widetilde{\eta}^{n,q}g'(\widetilde{\phi}_{i,1}^{n,q}), \quad & \text{if}~q=1;\\[1mm]
0,& \text{otherwise},\quad
\end{cases}
\quad \text{on}~\Gamma_{w},\label{eqs7}
\end{equation}
and with no-flux boundary conditions implemented by
\begin{equation}
\widetilde{J}_{\frac{1}{2},j}^{n,q} = 0,\quad \widetilde{J}_{N_{x}+\frac{1}{2},j}^{n,q} = 0,\quad j=1,2,\cdots,N_{y},\label{eqs8}
\end{equation}
where
\begin{equation}
(\partial_{x}\widetilde{\phi})_{i+\frac{1}{2},j}^{n,q}=\frac{\widetilde{\phi}_{i+1,j}^{n,q}-\widetilde{\phi}_{i,j}^{n,q}}{\Delta x}\quad \text{and} \quad (\partial_{y}\widetilde{\phi})_{i,j+\frac{1}{2}}^{n,q}=\frac{\widetilde{\phi}_{i,j+1}^{n,q}-\widetilde{\phi}_{i,j}^{n,q}}{\Delta y},\label{eqs9}
\end{equation}
and the Laplacian term $(\Delta\widetilde{\phi})_{i,j}^{n,q}$ is discretized by the following central difference formula
\begin{equation}
(\Delta\widetilde{\phi})_{i,j}^{n,q}=(\partial_{x}^{2}\widetilde{\phi})_{i,j}^{n,q}+(\partial_{y}^{2}\widetilde{\phi})_{i,j}^{n,q}=\frac{\widetilde{\phi}_{i+1,j}^{n,q}-2\widetilde{\phi}_{i,j}^{n,q}+\widetilde{\phi}_{i-1,j}^{n,q}}{\Delta x^{2}}
+\frac{\widetilde{\phi}_{i,j+1}^{n,q}-2\widetilde{\phi}_{i,j}^{n,q}+\widetilde{\phi}_{i,j-1}^{n,q}}{\Delta y^{2}}.\label{eqs10}
\end{equation}

The boundary conditions are imposed through the following ghost values:
\begin{equation}
\widetilde{\phi}_{0,j}^{n,q}=\widetilde{\phi}_{1,j}^{n,q},~
\widetilde{\phi}_{N_{x}+1,j}^{n,q}=\widetilde{\phi}_{N_{x},j}^{n,q},~
\widetilde{\phi}_{i,0}^{n,q}=\widetilde{\phi}_{i,1}^{n,q}-\frac{\Delta y}{\varepsilon^{2}}
\left(\frac{\widetilde{\phi}_{i,1}^{n,q}-\widetilde{\phi}_{i,1}^{n,q-1}}{\kappa\Delta t}
+\widetilde{\eta}^{n,q}g'(\widetilde{\phi}_{i,1}^{n,q})\right),~
\widetilde{\phi}_{i,N_{y}+1}^{n,q}=\widetilde{\phi}_{i,N_{y}}^{n,q}.
\label{eqghostq}
\end{equation}

Once the $x$-directional sweep is completed, the $y$-directional sweep is performed analogously,
with the initialization $\widehat{\phi}_{i,j}^{n,p}|_{p=0}=\widetilde{\phi}_{i,j}^{n,q}|_{q=N_{y}}$.
The scheme for each $y$-directional sweep reads:

$\mathbf{Step~2.~for~p=1,2,\cdots,N_{x}~do:}$
\begin{align}
&\widehat{\phi}_{i,j}^{n,p}-\widehat{\phi}_{i,j}^{n,p-1}=
\left\{\begin{array}{ll}
-\frac{\Delta t}{\Delta y}\left(\widehat{J}_{i,j+\frac{1}{2}}^{n,p}-\widehat{J}_{i,j-\frac{1}{2}}^{n,p}\right),& \text{if}~i=p;\\[1mm]
0,& \text{otherwise},
\end{array}\right.\label{eqs12}\\[1mm]
&\widehat{J}_{i,j+\frac{1}{2}}^{n,p}=\left(\widehat{V}_{i,j+\frac{1}{2}}^{n,p}\right)^{+}M(\widehat{\phi}_{i,j}^{n,p},\widehat{\phi}_{i,j+1}^{n,p})
+\left(\widehat{V}_{i,j+\frac{1}{2}}^{n,p}\right)^{-}M(\widehat{\phi}_{i,j+1}^{n,p},\widehat{\phi}_{i,j}^{n,p}),\label{eqs13}\\[1mm]
&\widehat{V}_{i,j+\frac{1}{2}}^{n,p}=-\frac{1}{\Delta y}\left(\widehat{\mu}_{i,j+1}^{n,p}-\widehat{\mu}_{i,j}^{n,p}\right),\label{eqs14}\\[2mm]
&\widehat{\mu}_{i,j}^{n,p}=-\varepsilon^{2}(\Delta\widehat{\phi})_{i,j}^{n,p}+\widehat{\xi}^{n,p}F'(\widehat{\phi}_{i,j}^{n,p}),\label{eqs15}\\[1mm]
&\sum_{i=1}^{N_{x}}\sum_{j=1}^{N_{y}}\left(F(\widehat{\phi}_{i,j}^{n,p})-F(\widehat{\phi}_{i,j}^{n,p-1})\right)=
\widehat{\xi}^{n,p}\sum_{i=1}^{N_{x}}\sum_{j=1}^{N_{y}}F'(\widehat{\phi}_{i,j}^{n,p})(\widehat{\phi}_{i,j}^{n,p}-\widehat{\phi}_{i,j}^{n,p-1}),\label{eqs16}\\[1mm]
&\sum_{i=1}^{N_{x}}\left(g(\widehat{\phi}_{i,1}^{n,p})-g(\widehat{\phi}_{i,1}^{n,p-1})\right)=\widehat{\eta}^{n,p}\sum_{i=1}^{N_{x}}g'(\widehat{\phi}_{i,1}^{n,p})\left(\widehat{\phi}_{i,1}^{n,p}-\widehat{\phi}_{i,1}^{n,p-1}\right),\label{eqs17}
\end{align}
subject to the following dynamic contact line boundary conditions
\begin{equation}
\frac{\widehat{\phi}_{i,1}^{n,p}-\widehat{\phi}_{i,1}^{n,p-1}}{\kappa\Delta t}=
\begin{cases}
\varepsilon^2 (\partial_{y}\widehat{\phi})_{i,\frac{1}{2}}^{n,p}-\widehat{\eta}^{n,p}g'(\widehat{\phi}_{i,1}^{n,p}), \quad
& \text{if~}i=p;\\[1mm]
0,& \text{otherwise},\quad
\end{cases}
\quad\text{on}~\Gamma_{w},\label{eqs18}
\end{equation}
and with no-flux boundary conditions implemented by
\begin{equation}
\widehat{J}_{i,\frac{1}{2}}^{n,p} = 0,\quad \widehat{J}_{i,N_{y}+\frac{1}{2}}^{n,p} = 0,\quad i=1,2,\cdots,N_{x},\label{eqs19}
\end{equation}
where
\begin{equation}
(\partial_{x}\widehat{\phi})_{i+\frac{1}{2},j}^{n,p}=\frac{\widehat{\phi}_{i+1,j}^{n,p}-\widehat{\phi}_{i,j}^{n,p}}{\Delta x}\quad \text{and} \quad (\partial_{y}\widehat{\phi})_{i,j+\frac{1}{2}}^{n,p}=\frac{\widehat{\phi}_{i,j+1}^{n,p}-\widehat{\phi}_{i,j}^{n,p}}{\Delta y},\label{eqs20}
\end{equation}
and the Laplacian term $(\Delta\widehat{\phi})_{i,j}^{n,p}$ is discretized as follows:
\begin{equation}
(\Delta\widehat{\phi})_{i,j}^{n,p}=(\partial_{x}^{2}\widehat{\phi})_{i,j}^{n,p}+(\partial_{y}^{2}\widehat{\phi})_{i,j}^{n,p}=\frac{\widehat{\phi}_{i+1,j}^{n,p}-2\widehat{\phi}_{i,j}^{n,p}+\widehat{\phi}_{i-1,j}^{n,p}}{\Delta x^{2}}
+\frac{\widehat{\phi}_{i,j+1}^{n,p}-2\widehat{\phi}_{i,j}^{n,p}+\widehat{\phi}_{i,j-1}^{n,p}}{\Delta y^{2}}.\label{eqs21}
\end{equation}
The ghost point values for the $y$-directional sweep are assigned analogously:
\begin{equation}
\widehat{\phi}_{0,j}^{n,p}=\widehat{\phi}_{1,j}^{n,p},~
\widehat{\phi}_{N_{x}+1,j}^{n,p}=\widehat{\phi}_{N_{x},j}^{n,p},~
\widehat{\phi}_{i,0}^{n,p}=\widehat{\phi}_{i,1}^{n,p}-\frac{\Delta y}{\varepsilon^{2}}
\left(\frac{\widehat{\phi}_{i,1}^{n,p}-\widehat{\phi}_{i,1}^{n,p-1}}{\kappa\Delta t}
+\widehat{\eta}^{n,p}g'(\widehat{\phi}_{i,1}^{n,p})\right),~
\widehat{\phi}_{i,N_{y}+1}^{n,p}=\widehat{\phi}_{i,N_{y}}^{n,p}.\label{eqghostp}
\end{equation}

Finally, when the above inner loops are completed,
the cell average and the Lagrange multipliers at the $(n+1)$-th step can be obtained as follows
\begin{equation}
\phi_{i,j}^{n+1}=\widehat{\phi}_{i,j}^{n,p}|_{p=_{N_{x}}},\quad
\xi^{n+1}=\frac{\sum_{q=1}^{N_{y}}\widetilde{\xi}^{n,q}+\sum_{p=1}^{N_{x}}\widehat{\xi}^{n,p}}{N_{y}+N_{x}}\quad\text{and}\quad
\eta^{n+1}=\frac{\sum_{q=1}^{N_{y}}\widetilde{\eta}^{n,q}+\sum_{p=1}^{N_{x}}\widehat{\eta}^{n,p}}{N_{y}+N_{x}}.\label{eqs23}
\end{equation}

The displayed full-step multipliers are arithmetic averages of the row and column values. The substep multipliers, rather than these averages, enter the update equations and the energy estimates.

The dimensional-splitting scheme constructed above retains the three structure-preserving
properties of the original upwind Lagrange multiplier scheme, namely boundedness, mass conservation, and energy
dissipation, as established by the following theorems.

\begin{thm}\label{thm4} (Boundedness)
The dimensional-splitting scheme \eqref{eqs1}-\eqref{eqs23} ensures the boundedness of the phase variable.
That is, for all $i,j$, if $|\phi_{i,j}^{n}|<1$, then $|\phi_{i,j}^{n+1}|<1$.
\end{thm}

\begin{proof}
We first prove that $|\phi_{i,j}^{n}|<1$ implies $|\widetilde{\phi}_{i,j}^{n,1}|<1$.
Suppose, to the contrary, that there exists a group of contiguous cells
$\{\widetilde{\phi}_{\alpha,j}^{n,1},\widetilde{\phi}_{\alpha+1,j}^{n,1},\cdots,\widetilde{\phi}_{k,j}^{n,1}\}$ with $\widetilde{\phi}_{i,j}^{n,1}\ge 1$.
Substituting $q=1$ into \eqref{eqs1} and summing over these cells yields
\begin{align}
0<&\sum_{i=\alpha}^{k}\frac{\Delta x}{\Delta t}\left(\widetilde{\phi}_{i,j}^{n,1}-\widetilde{\phi}_{i,j}^{n,q}|_{q=0}\right)=
-\sum_{i=\alpha}^{k}\left(\widetilde{J}_{i+\frac{1}{2},j}^{n,1}-\widetilde{J}_{i-\frac{1}{2},j}^{n,1}\right)=\widetilde{J}_{\alpha-\frac{1}{2},j}^{n,1}-\widetilde{J}_{k+\frac{1}{2},j}^{n,1}\nonumber\\
=&\,\left(\widetilde{V}_{\alpha-\frac{1}{2},j}^{n,1}\right)^{+}M(\widetilde{\phi}_{\alpha-1,j}^{n,1},\widetilde{\phi}_{\alpha,j}^{n,1})
+\left(\widetilde{V}_{\alpha-\frac{1}{2},j}^{n,1}\right)^{-}M(\widetilde{\phi}_{\alpha,j}^{n,1},\widetilde{\phi}_{\alpha-1,j}^{n,1})\nonumber\\
&-\left(\widetilde{V}_{k+\frac{1}{2},j}^{n,1}\right)^{+}M(\widetilde{\phi}_{k,j}^{n,1},\widetilde{\phi}_{k+1,j}^{n,1})
-\left(\widetilde{V}_{k+\frac{1}{2},j}^{n,1}\right)^{-}M(\widetilde{\phi}_{k+1,j}^{n,1},\widetilde{\phi}_{k,j}^{n,1}).\label{thb4}
\end{align}
By the definition~\eqref{eqn5} and the fact that $\widetilde{\phi}_{\alpha,j}^{n,1}\geq 1$ and $\widetilde{\phi}_{k,j}^{n,1}\geq 1$, we have
\begin{equation}
M(\widetilde{\phi}_{\alpha-1,j}^{n,1},\widetilde{\phi}_{\alpha,j}^{n,1})=0,\quad M(\widetilde{\phi}_{\alpha,j}^{n,1},\widetilde{\phi}_{\alpha-1,j}^{n,1})\geq0,\quad M(\widetilde{\phi}_{k,j}^{n,1},\widetilde{\phi}_{k+1,j}^{n,1})\geq0,\quad M(\widetilde{\phi}_{k+1,j}^{n,1},\widetilde{\phi}_{k,j}^{n,1})=0 .
\end{equation}
Therefore, the right-hand side of \eqref{thb4} must be non-positive, which contradicts the strict positivity of the left-hand side. Hence $\widetilde{\phi}_{i,j}^{n,1}<1$.
The same argument with reversed inequalities shows $\widetilde{\phi}_{i,j}^{n,1}>-1$. Repeating this argument for $q=2,3,\cdots,N_{y}$ yields $|\widetilde{\phi}_{i,j}^{n,q}|<1$ for all $q$.
An identical argument applied to the $y$-directional sweep gives $|\widehat{\phi}_{i,j}^{n,p}|<1$ for all $p$, since $\left|\widehat{\phi}_{i,j}^{n,p}|_{p=0}\right|=\left|\widetilde{\phi}_{i,j}^{n,q}|_{q=_{N_{y}}}\right|<1$,
and thus $|\phi_{i,j}^{n+1}|=\left|\widehat{\phi}_{i,j}^{n,p}|_{p=_{N_{x}}}\right|<1$.
\end{proof}

\begin{thm}\label{thm5} (Mass conservation)
The dimensional-splitting scheme \eqref{eqs1}-\eqref{eqs23} ensures that the total mass is conserved during the evolution, i.e.
\begin{equation}
m^{n+1}\triangleq\sum_{i=1}^{N_{x}}\sum_{j=1}^{N_{y}}\phi_{i,j}^{n+1}=\sum_{i=1}^{N_{x}}\sum_{j=1}^{N_{y}}\phi_{i,j}^{n}
=\cdots=\sum_{i=1}^{N_{x}}\sum_{j=1}^{N_{y}}\phi_{i,j}^{0}.\label{thd1}
\end{equation}
\end{thm}

\begin{proof}
Summing \eqref{eqs1} and \eqref{eqs12} over all cells $C_{i,j}$ gives
\begin{align}
\sum_{i=1}^{N_{x}}\sum_{j=1}^{N_{y}}(\widetilde{\phi}_{i,j}^{n,q}-\widetilde{\phi}_{i,j}^{n,q-1})
=-\frac{\Delta t}{\Delta x}\sum_{i=1}^{N_{x}}\left(\widetilde{J}_{i+\frac{1}{2},q}^{n,q}-\widetilde{J}_{i-\frac{1}{2},q}^{n,q}\right)
=-\frac{\Delta t}{\Delta x}\left(\widetilde{J}_{N_{x}+\frac{1}{2},q}^{n,q}-\widetilde{J}_{{\frac{1}{2},q}}^{n,q}\right)=0,
\quad \forall q,\label{thd2}\\
\sum_{i=1}^{N_{x}}\sum_{j=1}^{N_{y}}(\widehat{\phi}_{i,j}^{n,p}-\widehat{\phi}_{i,j}^{n,p-1})
=-\frac{\Delta t}{\Delta y}\sum_{j=1}^{N_{y}}\left(\widehat{J}_{p,j+\frac{1}{2}}^{n,p}-\widehat{J}_{p,j-\frac{1}{2}}^{n,p}\right)
=-\frac{\Delta t}{\Delta y}\left(\widehat{J}_{p,N_{y}+\frac{1}{2}}^{n,p}-\widehat{J}_{{p,\frac{1}{2}}}^{n,p}\right)=0,
\quad \forall p.\label{thd3}
\end{align}
Therefore,
\begin{align}
\sum_{i=1}^{N_{x}}\sum_{j=1}^{N_{y}}\phi_{i,j}^{n+1}
=&\sum_{i=1}^{N_{x}}\sum_{j=1}^{N_{y}}\widehat{\phi}_{i,j}^{n,p}|_{p=_{N_{x}}}
=\sum_{i=1}^{N_{x}}\sum_{j=1}^{N_{y}}\widehat{\phi}_{i,j}^{n,p}|_{p=_{N_{x}}-1}
=\cdots=\sum_{i=1}^{N_{x}}\sum_{j=1}^{N_{y}}\widehat{\phi}_{i,j}^{n,p}|_{p=0}\nonumber\\
=&\sum_{i=1}^{N_{x}}\sum_{j=1}^{N_{y}}\widetilde{\phi}_{i,j}^{n,q}|_{q=_{N_{y}}}
=\sum_{i=1}^{N_{x}}\sum_{j=1}^{N_{y}}\widetilde{\phi}_{i,j}^{n,q}|_{q=_{N_{y}}-1}
=\cdots=\sum_{i=1}^{N_{x}}\sum_{j=1}^{N_{y}}\widetilde{\phi}_{i,j}^{n,q}|_{q=0}\nonumber\\
=&\sum_{i=1}^{N_{x}}\sum_{j=1}^{N_{y}}\phi_{i,j}^{n}
=\cdots=\sum_{i=1}^{N_{x}}\sum_{j=1}^{N_{y}}\phi_{i,j}^{0},\label{thd4}
\end{align}
where the initial conditions $\widehat{\phi}_{i,j}^{n,p}|_{p=0}=\widetilde{\phi}_{i,j}^{n,q}|_{q=_{N_{y}}}$ and $\widetilde{\phi}_{i,j}^{n,q}|_{q=0}=\phi_{i,j}^{n}$ are applied.
\end{proof}

\begin{thm}\label{thm6} (Energy dissipation)
The dimensional-splitting scheme \eqref{eqs1}-\eqref{eqs23} is energy stable,
and satisfies the following discrete energy dissipation law:
\begin{align}
\frac{\mathcal{W}^{n+1}-\mathcal{W}^{n}}{\Delta t}
\le&\, -\Delta x\Delta y\sum_{p=1}^{N_{x}}\sum_{j=1}^{N_{y}-1}
\min\left\{M(\widehat{\phi}_{p,j}^{n,p},\widehat{\phi}_{p,j+1}^{n,p}),M(\widehat{\phi}_{p,j+1}^{n,p},\widehat{\phi}_{p,j}^{n,p})\right\}
\left|\widehat{V}_{p,j+\frac{1}{2}}^{n,p}\right|^{2}\nonumber\\
&\,-\Delta x\Delta y\sum_{i=1}^{N_{x}-1}\sum_{q=1}^{N_{y}}
\min\left\{M(\widetilde{\phi}_{i,q}^{n,q},\widetilde{\phi}_{i+1,q}^{n,q}),M(\widetilde{\phi}_{i+1,q}^{n,q},\widetilde{\phi}_{i,q}^{n,q})\right\}
\left|\widetilde{V}_{i+\frac{1}{2},q}^{n,q}\right|^{2}
\le\, 0,\label{thf1}
\end{align}
where
\begin{align}
\mathcal{W}^{n}
=&\,\Delta x\Delta y\sum_{i=1}^{N_{x}-1}\sum_{j=1}^{N_{y}}\frac{\varepsilon^{2}}{2}\left(\frac{\phi_{i+1,j}^{n}-\phi_{i,j}^{n}}{\Delta x}\right)^{2}
+\Delta x\Delta y\sum_{i=1}^{N_{x}}\sum_{j=1}^{N_{y}-1}\frac{\varepsilon^{2}}{2}\left(\frac{\phi_{i,j+1}^{n}-\phi_{i,j}^{n}}{\Delta y}\right)^{2}
\nonumber\\
&+\Delta x\Delta y\sum_{i=1}^{N_{x}}\sum_{j=1}^{N_{y}}F(\phi_{i,j}^{n})+\Delta x\sum_{i=1}^{N_{x}}g(\phi_{i,1}^{n}).\label{thf2}
\end{align}
\end{thm}

\begin{proof}
We first establish energy dissipation for each $x$-directional sweep:
\begin{equation}
\frac{\widetilde{\mathcal{W}}^{n,q}-\widetilde{\mathcal{W}}^{n,q-1}}{\Delta t}
\le -\Delta x\Delta y\sum_{i=1}^{N_{x}-1}
\min\left\{M(\widetilde{\phi}_{i,q}^{n,q},\widetilde{\phi}_{i+1,q}^{n,q}),M(\widetilde{\phi}_{i+1,q}^{n,q},\widetilde{\phi}_{i,q}^{n,q})\right\}
\left|\widetilde{V}_{i+\frac{1}{2},q}^{n,q}\right|^{2}\le 0,\label{thf3}
\end{equation}
where
\begin{align}
\widetilde{\mathcal{W}}^{n,q}
=&\,\Delta x\Delta y\sum_{i=1}^{N_{x}-1}\sum_{j=1}^{N_{y}}\frac{\varepsilon^{2}}{2}
\left(\frac{\widetilde{\phi}_{i+1,j}^{n,q}-\widetilde{\phi}_{i,j}^{n,q}}{\Delta x}\right)^{2}
+\Delta x\Delta y\sum_{i=1}^{N_{x}}\sum_{j=1}^{N_{y}-1}\frac{\varepsilon^{2}}{2}
\left(\frac{\widetilde{\phi}_{i,j+1}^{n,q}-\widetilde{\phi}_{i,j}^{n,q}}{\Delta y}\right)^{2}\nonumber\\
&+\Delta x\Delta y\sum_{i=1}^{N_{x}}\sum_{j=1}^{N_{y}}F(\widetilde{\phi}_{i,j}^{n,q})+\Delta x\sum_{i=1}^{N_{x}}g(\widetilde{\phi}_{i,1}^{n,q}).\label{thf4}
\end{align}
Subtracting the first term on the right-hand side of \eqref{thf4} at consecutive inner-loop steps and following the same procedure as in the proof of Theorem~\ref{thm3}, we obtain
\begin{align}
&\,\frac{\varepsilon^{2}}{2}\sum_{i=1}^{N_{x}-1}\sum_{j=1}^{N_{y}}\left[\left(\frac{\widetilde{\phi}_{i+1,j}^{n,q}-\widetilde{\phi}_{i,j}^{n,q}}{\Delta x}\right)^{2}
-\left(\frac{\widetilde{\phi}_{i+1,j}^{n,q-1}-\widetilde{\phi}_{i,j}^{n,q-1}}{\Delta x}\right)^{2}\right]\nonumber\\
=&\,-\varepsilon^{2}\sum_{i=2}^{N_{x}-1}\sum_{j=1}^{N_{y}}
\left(\frac{\widetilde{\phi}_{i+1,j}^{n,q}-2\widetilde{\phi}_{i,j}^{n,q}+\widetilde{\phi}_{i-1,j}^{n,q}}{\Delta x^{2}}\right)\cdot(\widetilde{\phi}_{i,j}^{n,q}-\widetilde{\phi}_{i,j}^{n,q-1})\nonumber\\
&\,+\varepsilon^{2}\sum_{j=1}^{N_{y}}\left(\frac{\widetilde{\phi}_{N_{x},j}^{n,q}-\widetilde{\phi}_{N_{x}-1,j}^{n,q}}{\Delta x^{2}}\right)
\cdot(\widetilde{\phi}_{N_{x},j}^{n,q}-\widetilde{\phi}_{N_{x},j}^{n,q-1})
-\varepsilon^{2}\sum_{j=1}^{N_{y}}\left(\frac{\widetilde{\phi}_{2,j}^{n,q}-\widetilde{\phi}_{1,j}^{n,q}}{\Delta x^{2}}\right)
\cdot(\widetilde{\phi}_{1,j}^{n,q}-\widetilde{\phi}_{1,j}^{n,q-1})\nonumber\\
&\,-\frac{\varepsilon^{2}}{2}\sum_{i=1}^{N_{x}-1}\sum_{j=1}^{N_{y}}\left[\left(\frac{\widetilde{\phi}_{i+1,j}^{n,q}-\widetilde{\phi}_{i,j}^{n,q}}{\Delta x}\right)
-\left(\frac{\widetilde{\phi}_{i+1,j}^{n,q-1}-\widetilde{\phi}_{i,j}^{n,q-1}}{\Delta x}\right)\right]^{2}\nonumber\\
=&\,-\varepsilon^{2}\sum_{i=1}^{N_{x}}\sum_{j=1}^{N_{y}}(\partial_{x}^{2}\widetilde{\phi})_{i,j}^{n,q}\cdot(\widetilde{\phi}_{i,j}^{n,q}-\widetilde{\phi}_{i,j}^{n,q-1})-\frac{\varepsilon^{2}}{2}\sum_{i=1}^{N_{x}-1}\sum_{j=1}^{N_{y}}\left[\left(\frac{\widetilde{\phi}_{i+1,j}^{n,q}-\widetilde{\phi}_{i,j}^{n,q}}{\Delta x}\right)
-\left(\frac{\widetilde{\phi}_{i+1,j}^{n,q-1}-\widetilde{\phi}_{i,j}^{n,q-1}}{\Delta x}\right)\right]^{2}.
\label{thf5}
\end{align}
A similar computation for the $y$-derivative term gives
\begin{align}
&\,\frac{\varepsilon^{2}}{2}\sum_{i=1}^{N_{x}}\sum_{j=1}^{N_{y}-1}\left[\left(\frac{\widetilde{\phi}_{i,j+1}^{n,q}-\widetilde{\phi}_{i,j}^{n,q}}{\Delta y}\right)^{2}
-\left(\frac{\widetilde{\phi}_{i,j+1}^{n,q-1}-\widetilde{\phi}_{i,j}^{n,q-1}}{\Delta y}\right)^{2}\right]\nonumber\\
=&\,-\varepsilon^{2}\sum_{i=1}^{N_{x}}\sum_{j=2}^{N_{y}-1}
\left(\frac{\widetilde{\phi}_{i,j+1}^{n,q}-2\widetilde{\phi}_{i,j}^{n,q}+\widetilde{\phi}_{i,j-1}^{n,q}}{\Delta y^{2}}\right)\cdot(\widetilde{\phi}_{i,j}^{n,q}-\widetilde{\phi}_{i,j}^{n,q-1})\nonumber\\
&+\varepsilon^{2}\sum_{i=1}^{N_{x}}\left(\frac{\widetilde{\phi}_{i,N_{y}}^{n,q}-\widetilde{\phi}_{i,N_{y}-1}^{n,q}}{\Delta y^{2}}\right)
\cdot(\widetilde{\phi}_{i,N_{y}}^{n,q}-\widetilde{\phi}_{i,N_{y}}^{n,q-1})
-\varepsilon^{2}\sum_{i=1}^{N_{x}}\left(\frac{\widetilde{\phi}_{i,2}^{n,q}-\widetilde{\phi}_{i,1}^{n,q}}{\Delta y^{2}}\right)\cdot(\widetilde{\phi}_{i,1}^{n,q}-\widetilde{\phi}_{i,1}^{n,q-1})
\nonumber\\
&-\frac{\varepsilon^{2}}{2}\sum_{i=1}^{N_{x}}\sum_{j=1}^{N_{y}-1}\left[\left(\frac{\widetilde{\phi}_{i,j+1}^{n,q}-\widetilde{\phi}_{i,j}^{n,q}}{\Delta y}\right)
-\left(\frac{\widetilde{\phi}_{i,j+1}^{n,q-1}-\widetilde{\phi}_{i,j}^{n,q-1}}{\Delta y}\right)\right]^{2}\nonumber\\
=&\,-\varepsilon^{2}\sum_{i=1}^{N_{x}}\sum_{j=1}^{N_{y}}(\partial_{y}^{2}\widetilde{\phi})_{i,j}^{n,q}\cdot(\widetilde{\phi}_{i,j}^{n,q}-\widetilde{\phi}_{i,j}^{n,q-1})	-\varepsilon^{2}\sum_{i=1}^{N_{x}}\left(\frac{\widetilde{\phi}_{i,1}^{n,q}-\widetilde{\phi}_{i,0}^{n,q}}{\Delta y^{2}}\right)\cdot(\widetilde{\phi}_{i,1}^{n,q}-\widetilde{\phi}_{i,1}^{n,q-1})\nonumber\\
&-\frac{\varepsilon^{2}}{2}\sum_{i=1}^{N_{x}}\sum_{j=1}^{N_{y}-1}\left[\left(\frac{\widetilde{\phi}_{i,j+1}^{n,q}-\widetilde{\phi}_{i,j}^{n,q}}{\Delta y}\right)
-\left(\frac{\widetilde{\phi}_{i,j+1}^{n,q-1}-\widetilde{\phi}_{i,j}^{n,q-1}}{\Delta y}\right)\right]^{2}.
\label{thf6}
\end{align}

Multiplying \eqref{eqs1} by $\widetilde{\mu}_{i,j}^{n,q}$ and summing over the cells gives
\begin{align}
&\,\sum_{i=1}^{N_{x}}\sum_{j=1}^{N_{y}}(\widetilde{\phi}_{i,j}^{n,q}-\widetilde{\phi}_{i,j}^{n,q-1})\widetilde{\mu}_{i,j}^{n,q} =-\sum_{i=1}^{N_{x}}\frac{\Delta t}{\Delta x}\left(\widetilde{J}_{i+\frac{1}{2},q}^{n,q}-\widetilde{J}_{i-\frac{1}{2},q}^{n,q}\right)\widetilde{\mu}_{i,q}^{n,q}
=\sum_{i=1}^{N_{x}-1}\frac{\Delta t}{\Delta x}\widetilde{J}_{i+\frac{1}{2},q}^{n,q}\left(\widetilde{\mu}_{i+1,q}^{n,q}-\widetilde{\mu}_{i,q}^{n,q}\right)\nonumber\\
=&\,-\Delta t\sum_{i=1}^{N_{x}-1}\widetilde{J}_{i+\frac{1}{2},q}^{n,q}\widetilde{V}_{i+\frac{1}{2},q}^{n,q}
=-\Delta t\sum_{i=1}^{N_{x}-1}\left[\left(\widetilde{V}_{i+\frac{1}{2},q}^{n,q}\right)^{+}M(\widetilde{\phi}_{i,q}^{n,q},\widetilde{\phi}_{i+1,q}^{n,q})
+\left(\widetilde{V}_{i+\frac{1}{2},q}^{n,q}\right)^{-}M(\widetilde{\phi}_{i+1,q}^{n,q},\widetilde{\phi}_{i,q}^{n,q})\right]\widetilde{V}_{i+\frac{1}{2},q}^{n,q}\nonumber\\
\le&\,-\Delta t\sum_{i=1}^{N_{x}-1}\min\left\{M(\widetilde{\phi}_{i,q}^{n,q},\widetilde{\phi}_{i+1,q}^{n,q}),M(\widetilde{\phi}_{i+1,q}^{n,q},\widetilde{\phi}_{i,q}^{n,q})\right\}
\left|\widetilde{V}_{i+\frac{1}{2},q}^{n,q}\right|^{2}\le 0.\label{thf7}
\end{align}

Combining~\eqref{thf5}--\eqref{thf7} with the row-sweep ghost relation \eqref{eqghostq} gives
\begin{align}
\frac{\widetilde{\mathcal{W}}^{n,q}-\widetilde{\mathcal{W}}^{n,q-1}}{\Delta x\Delta y}
=&\,\frac{\varepsilon^{2}}{2}\sum_{i=1}^{N_{x}-1}\sum_{j=1}^{N_{y}}\left[\left(\frac{\widetilde{\phi}_{i+1,j}^{n,q}-\widetilde{\phi}_{i,j}^{n,q}}{\Delta x}\right)^{2}
-\left(\frac{\widetilde{\phi}_{i+1,j}^{n,q-1}-\widetilde{\phi}_{i,j}^{n,q-1}}{\Delta x}\right)^{2}\right]\nonumber\\
&\,+\frac{\varepsilon^{2}}{2}\sum_{i=1}^{N_{x}}\sum_{j=1}^{N_{y}-1}\left[\left(\frac{\widetilde{\phi}_{i,j+1}^{n,q}-\widetilde{\phi}_{i,j}^{n,q}}{\Delta y}\right)^{2}
-\left(\frac{\widetilde{\phi}_{i,j+1}^{n,q-1}-\widetilde{\phi}_{i,j}^{n,q-1}}{\Delta y}\right)^{2}\right]\nonumber\\
&\,+\sum_{i=1}^{N_{x}}\sum_{j=1}^{N_{y}}\left[F(\widetilde{\phi}_{i,j}^{n,q})-F(\widetilde{\phi}_{i,j}^{n,q-1})\right]+\frac{1}{\Delta y}\sum_{i=1}^{N_{x}}\left[g(\widetilde{\phi}_{i,1}^{n,q})-g(\widetilde{\phi}_{i,1}^{n,q-1})\right]\nonumber\\
\le&\,-\varepsilon^{2}\sum_{i=1}^{N_{x}}\sum_{j=1}^{N_{y}}\left[(\partial_{x}^{2}\widetilde{\phi})_{i,j}^{n,q}+(\partial_{y}^{2}\widetilde{\phi})_{i,j}^{n,q}\right](\widetilde{\phi}_{i,j}^{n,q}-\widetilde{\phi}_{i,j}^{n,q-1})\nonumber\\
&\,-\varepsilon^{2}\sum_{i=1}^{N_{x}}\left(\frac{\widetilde{\phi}_{i,1}^{n,q}-\widetilde{\phi}_{i,0}^{n,q}}{\Delta y^{2}}\right)(\widetilde{\phi}_{i,1}^{n,q}-\widetilde{\phi}_{i,1}^{n,q-1})\nonumber\\
&\,+\widetilde{\xi}^{n,q}\sum_{i=1}^{N_{x}}\sum_{j=1}^{N_{y}}F'(\widetilde{\phi}_{i,j}^{n,q})(\widetilde{\phi}_{i,j}^{n,q}-\widetilde{\phi}_{i,j}^{n,q-1})+\frac{1}{\Delta y}\widetilde{\eta}^{n,q}\sum_{i=1}^{N_{x}}g'(\widetilde{\phi}_{i,1}^{n,q})\left(\widetilde{\phi}_{i,1}^{n,q}-\widetilde{\phi}_{i,1}^{n,q-1}\right)\nonumber\\
\le&\,-\varepsilon^{2}\sum_{i=1}^{N_{x}}\sum_{j=1}^{N_{y}}(\Delta\widetilde{\phi})_{i,j}^{n,q}(\widetilde{\phi}_{i,j}^{n,q}-\widetilde{\phi}_{i,j}^{n,q-1})-\frac{1}{\kappa\Delta t\Delta y}\sum_{i=1}^{N_{x}}\left(\widetilde{\phi}_{i,1}^{n,q}-\widetilde{\phi}_{i,1}^{n,q-1}\right)^{2}\nonumber\\
&\,+\widetilde{\xi}^{n,q}\sum_{i=1}^{N_{x}}\sum_{j=1}^{N_{y}}F'(\widetilde{\phi}_{i,j}^{n,q})(\widetilde{\phi}_{i,j}^{n,q}-\widetilde{\phi}_{i,j}^{n,q-1})\nonumber\\
\le&\,-\varepsilon^{2}\sum_{i=1}^{N_{x}}\sum_{j=1}^{N_{y}}(\Delta\widetilde{\phi})_{i,j}^{n,q}(\widetilde{\phi}_{i,j}^{n,q}-\widetilde{\phi}_{i,j}^{n,q-1})+\widetilde{\xi}^{n,q}\sum_{i=1}^{N_{x}}\sum_{j=1}^{N_{y}}F'(\widetilde{\phi}_{i,j}^{n,q})(\widetilde{\phi}_{i,j}^{n,q}-\widetilde{\phi}_{i,j}^{n,q-1})\nonumber\\
=&\,\sum_{i=1}^{N_{x}}\sum_{j=1}^{N_{y}}(\widetilde{\phi}_{i,j}^{n,q}-\widetilde{\phi}_{i,j}^{n,q-1})\widetilde{\mu}_{i,j}^{n,q}\nonumber\\
\le&\,-\Delta t\sum_{i=1}^{N_{x}-1}\min\left\{M(\widetilde{\phi}_{i,q}^{n,q},\widetilde{\phi}_{i+1,q}^{n,q}),M(\widetilde{\phi}_{i+1,q}^{n,q},\widetilde{\phi}_{i,q}^{n,q})\right\}
\left|\widetilde{V}_{i+\frac{1}{2},q}^{n,q}\right|^{2}\le 0.\label{thf8}
\end{align}

An analogous argument shows that each $y$-directional sweep is also dissipative:
\begin{equation}
\frac{\widehat{\mathcal{W}}^{n,p}-\widehat{\mathcal{W}}^{n,p-1}}{\Delta t}
\le -\Delta x\Delta y\sum_{j=1}^{N_{y}-1}\min\left\{M(\widehat{\phi}_{p,j}^{n,p},\widehat{\phi}_{p,j+1}^{n,p}),M(\widehat{\phi}_{p,j+1}^{n,p},\widehat{\phi}_{p,j}^{n,p})\right\}
\left|\widehat{V}_{p,j+\frac{1}{2}}^{n,p}\right|^{2}\le 0,\label{thf9}
\end{equation}
where
\begin{align}
\widehat{\mathcal{W}}^{n,p}
=&\,\Delta x\Delta y\sum_{i=1}^{N_{x}-1}\sum_{j=1}^{N_{y}}\frac{\varepsilon^{2}}{2}\left(\frac{\widehat{\phi}_{i+1,j}^{n,p}-\widehat{\phi}_{i,j}^{n,p}}{\Delta x}\right)^{2}
+\Delta x\Delta y\sum_{i=1}^{N_{x}}\sum_{j=1}^{N_{y}-1}\frac{\varepsilon^{2}}{2}\left(\frac{\widehat{\phi}_{i,j+1}^{n,p}-\widehat{\phi}_{i,j}^{n,p}}{\Delta y}\right)^{2}\nonumber\\
&+\Delta x\Delta y\sum_{i=1}^{N_{x}}\sum_{j=1}^{N_{y}}F(\widehat{\phi}_{i,j}^{n,p})+\Delta x\sum_{i=1}^{N_{x}}g(\widehat{\phi}_{i,1}^{n,p}).\label{thf10}
\end{align}

Summing the contributions from both sweeps yields
\begin{align}
\frac{\mathcal{W}^{n+1}-\mathcal{W}^{n}}{\Delta t}=&\,\frac{\widehat{\mathcal{W}}^{n,p}|_{p=_{N_{x}}}-\widetilde{\mathcal{W}}^{n,q}|_{q=0}
}{\Delta t}=\frac{\widehat{\mathcal{W}}^{n,p}|_{p=_{N_{x}}}-\widehat{\mathcal{W}}^{n,p}|_{p=0}
+\widetilde{\mathcal{W}}^{n,q}|_{q=_{N_{y}}}-\widetilde{\mathcal{W}}^{n,q}|_{q=0}}{\Delta t}\nonumber\\
=&\,\sum_{p=1}^{N_{x}}\frac{\widehat{\mathcal{W}}^{n,p}-\widehat{\mathcal{W}}^{n,p-1}}{\Delta t}
+\sum_{q=1}^{N_{y}}\frac{\widetilde{\mathcal{W}}^{n,q}-\widetilde{\mathcal{W}}^{n,q-1}}{\Delta t}\nonumber\\
\le&\, -\Delta x\Delta y\sum_{p=1}^{N_{x}}\sum_{j=1}^{N_{y}-1}
\min\left\{M(\widehat{\phi}_{p,j}^{n,p},\widehat{\phi}_{p,j+1}^{n,p}),M(\widehat{\phi}_{p,j+1}^{n,p},\widehat{\phi}_{p,j}^{n,p})\right\}\left|\widehat{V}_{p,j+\frac{1}{2}}^{n,p}\right|^{2}\nonumber\\
&-\Delta x\Delta y\sum_{i=1}^{N_{x}-1}\sum_{q=1}^{N_{y}}
\min\left\{M(\widetilde{\phi}_{i,q}^{n,q},\widetilde{\phi}_{i+1,q}^{n,q}),M(\widetilde{\phi}_{i+1,q}^{n,q},\widetilde{\phi}_{i,q}^{n,q})\right\}\left|\widetilde{V}_{i+\frac{1}{2},q}^{n,q}\right|^{2}
\le 0.\label{thf11}
\end{align}

This completes the proof.
\end{proof}

Thus, the sequential splitting replaces a multidimensional nonlinear update with a sequence of one-dimensional solves while preserving the three structural properties. The wall ghost relations \eqref{eqghostq} and \eqref{eqghostp} retain the coupling to contact line relaxation. For an idealized cost comparison, suppose that solving a system with $N$ unknowns costs $\mathcal{O}(N^{\beta})$, where $2<\beta\le3$, and that iteration counts are comparable. With $N$ cells per direction in $d$ dimensions, a monolithic solve then costs $\mathcal{O}(N^{d\beta})$, whereas $dN^{d-1}$ line solves cost $\mathcal{O}(dN^{\beta+d-1})$. This estimate explains the potential saving from smaller systems; actual runtimes also depend on sparsity, solver design, and nonlinear convergence.

\section{Numerical results}
\label{sec5}

The analysis above establishes discrete structural properties; the experiments now examine the resulting evolution and its sensitivity to modeling choices. Using the dimensional-splitting scheme, we first compare single-film shrinkage with the estimate in Section~2.3 and then examine mass transfer between separated islands. We next study contact line relaxation and wettability-dependent shapes, before considering an elongated film whose evolution may involve breakup. Throughout, bounds, mass, and energy diagnostics complement the geometric observations.

The initial phase-field profile is chosen as
\begin{equation}
\phi(\bm{x},t)|_{t=0}=\beta_{_{\theta}}\tanh\left(\frac{\mathrm{dist}(\bm{x},\Gamma_{0})}{\sqrt{2}\varepsilon}\right),\quad\bm{x}\in\Omega,\label{eqz1}
\end{equation}

where $\Gamma_{0}$ denotes the initial curve or surface separating the film and vapor phases, and $\mathrm{dist}(\bm{x},\Gamma_{0})$ denotes the signed distance from point $\bm{x}$ to $\Gamma_{0}$, taken positive inside the film. For the polynomial potential, the prefactor $\beta_{\theta}$ is replaced by $1$. The hyperbolic-tangent profile provides smooth initial data; it is not an exact equilibrium profile for the logarithmic potential. A Newton-type iteration solves the nonlinear systems, using the previous time-step solution as the initial guess. Unless specified otherwise, the parameters are those in Table~\ref{tabp}.

\begin{table}[!htp]
\renewcommand\arraystretch{1.5}
\centering
\caption{Default parameter values used in the numerical simulations.}\label{tabp}
\tabcolsep 0.40in
\begin{tabular}[c]
{cccccc} \hline
$\Delta t$   &$\Delta x$   &$\Delta y$   &$\varepsilon$   &$\kappa$   &$\theta_{s}$\\   \hline
$10^{-4}$   &$0.004$   &$0.004$   &$0.02$   &$5000$   &$\frac{3\pi}{4}$\\   \hline
\end{tabular}
\end{table}

\subsection{Spontaneous shrinkage}

We begin with a single circular film segment, for which the estimate in Section~2.3 provides a quantitative reference for finite-width shrinkage. The domain is $\Omega=[-2,2]\times[0,3]$, and the initial segment has radius $r_{_0}=1$ and Young angle $\theta_s=3\pi/4$. After the solution becomes numerically stationary, the radius $r_{\rm eq}$ is obtained by fitting a circular arc to the $\phi=0$ contour. Figure~\ref{figsdr} compares $\delta r=r_{\rm eq}-r_{_0}$ with \eqref{eqt7} for $\theta\in[0.2,0.9]$. The computation reproduces the predicted increase in contraction with temperature. Agreement is close at lower temperatures, while a visible discrepancy develops toward the upper end of the tested range. The comparison supports the estimate as a leading-order description within its stated assumptions, rather than an exact prediction at all temperatures.

\begin{figure}[!htp]
	\centering
	\includegraphics[width=8cm]{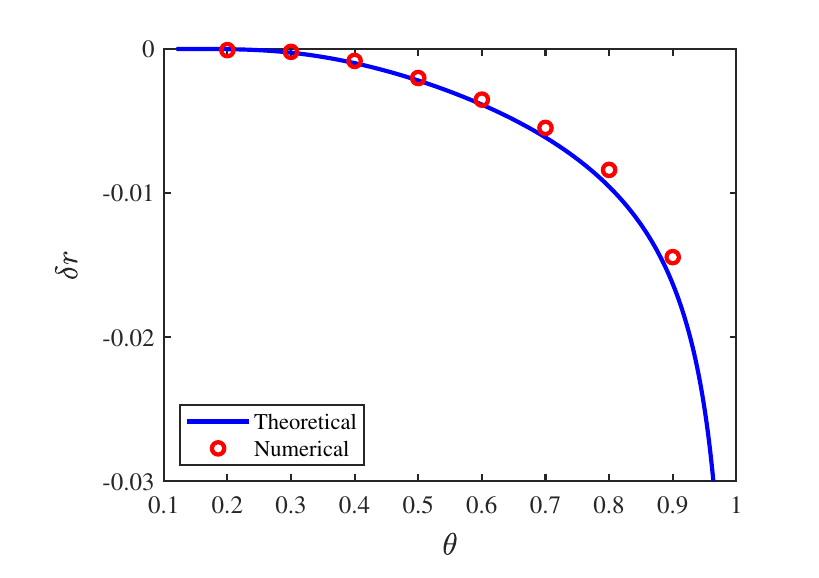}
	\caption{
Predicted and computed radius contraction $\delta r$ for the logarithmic Flory--Huggins potential as a function of $\theta$.}\label{figsdr}
\end{figure}

\subsection{Coarsening effect}

The single-film test measures geometric shrinkage, but it does not reveal whether material is transferred between disconnected islands. We now examine this second finite-width effect on a substrate, where bulk-diffusion coarsening can occur even when the mobility is degenerate~\cite{Pesce21,Bretin22}. Following the motivation of the substrate-free tests in~\cite{Huang23b}, we compare the logarithmic and polynomial potentials in the presence of dynamic contact lines.

The initial configuration consists of a square of side length $0.4$ centered at $(0,0.2)$ and two squares of side length $0.2$ centered at $(\pm0.45,0.1)$. Figure~\ref{figsdiffusion} compares their evolution for several logarithmic-potential temperatures and for the polynomial potential. At the higher temperatures shown, and with the polynomial potential, the smaller islands shrink while the central island grows. At $\theta=0.3$, all three islands remain distinct over the reported interval and relax toward rounded shapes, indicating substantially reduced inter-island mass transfer.

Table~\ref{tabS} reports the relative change in total film area, $\delta S=|S(t)-S(0)|/S(0)$, at $t=5$. It is $0.1578\%$ for $\theta=0.3$, compared with $1.8791\%$ for the polynomial potential. This diagnostic measures geometric area variation, not the error in the conserved phase-field mass, and does not by itself measure exchange between individual islands. Together with the snapshots, it indicates reduced area loss and coarsening at low temperature. To place these morphological differences alongside the discrete structural behavior, Figures~\ref{figsdiffusion2} and~\ref{figsdiffusion3} report phase bounds, energy, and mass. Both computations show bounded phase values, decreasing energy, and conserved mass to the resolution of the plots. The multipliers approach $1$ after initial transients; the logarithmic case exhibits a visible early deviation in the wall multiplier.

\begin{figure}[htp]
\centering
\includegraphics[width=\textwidth]{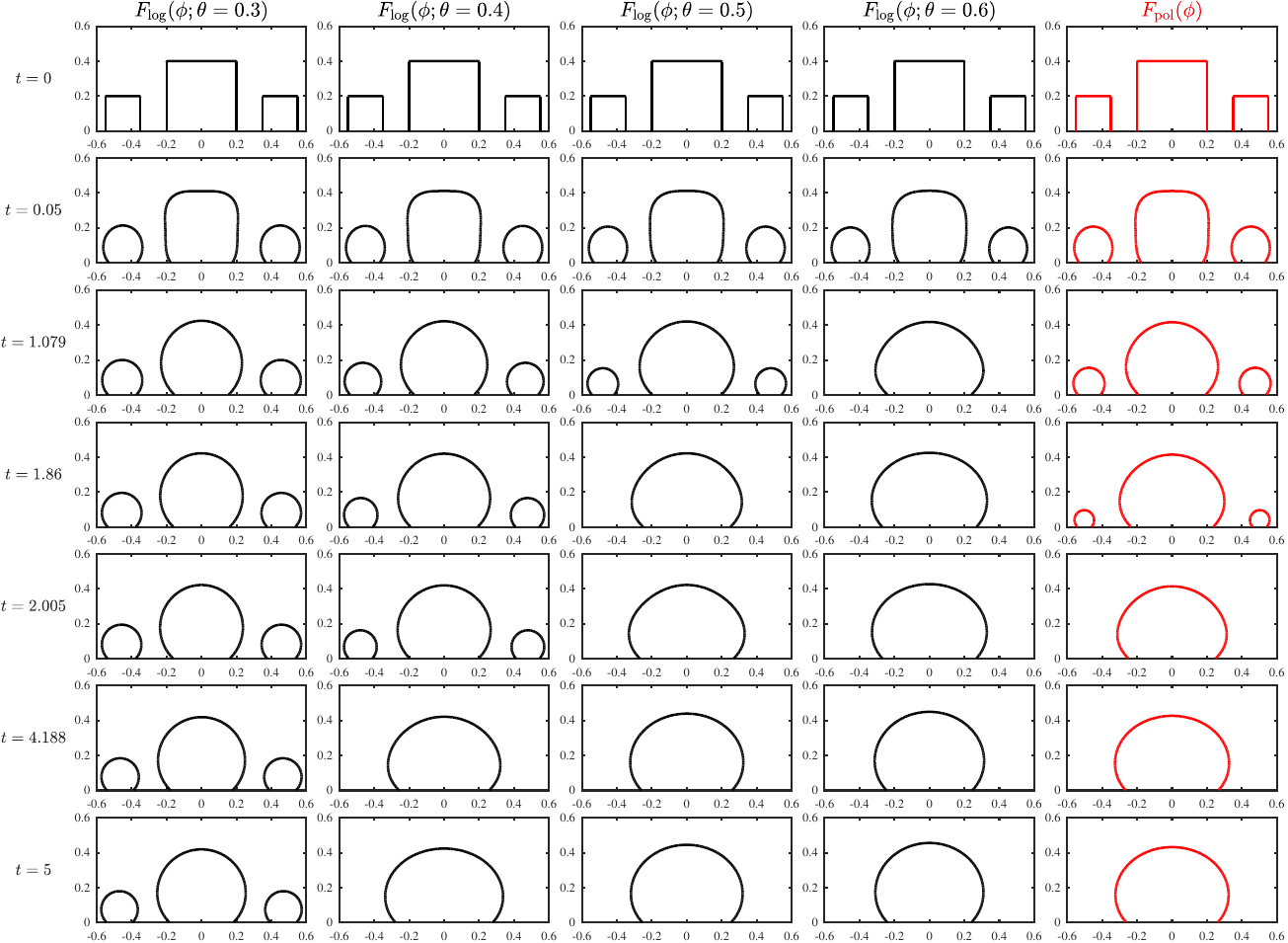}
\caption{Evolution of a large square film and two smaller squares under the logarithmic potential at various temperatures $\theta$ and under the polynomial potential.}\label{figsdiffusion}
\end{figure}

\begin{figure}[!htp]
\centering
\includegraphics[width=\textwidth]{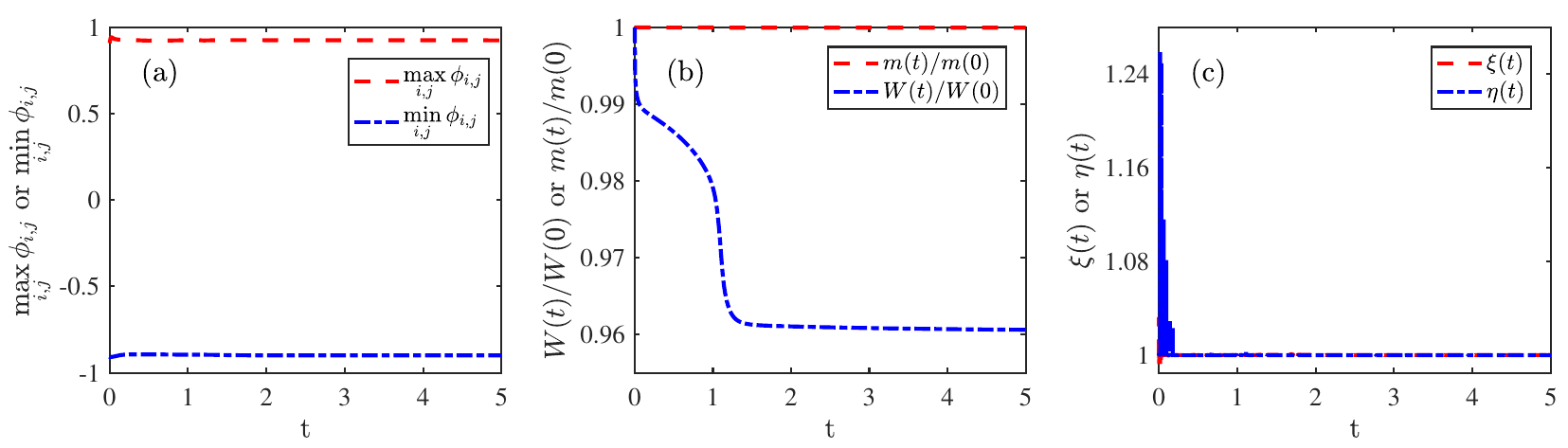}
\caption{
Time histories of (a) the extrema of $\phi_{i,j}$, (b) normalized energy and mass, and (c) the Lagrange multipliers $\xi(t)$ and $\eta(t)$ under the logarithmic potential at $\theta=0.6$.}\label{figsdiffusion2}
\end{figure}
\begin{figure}[!htp]
\centering
\includegraphics[width=\textwidth]{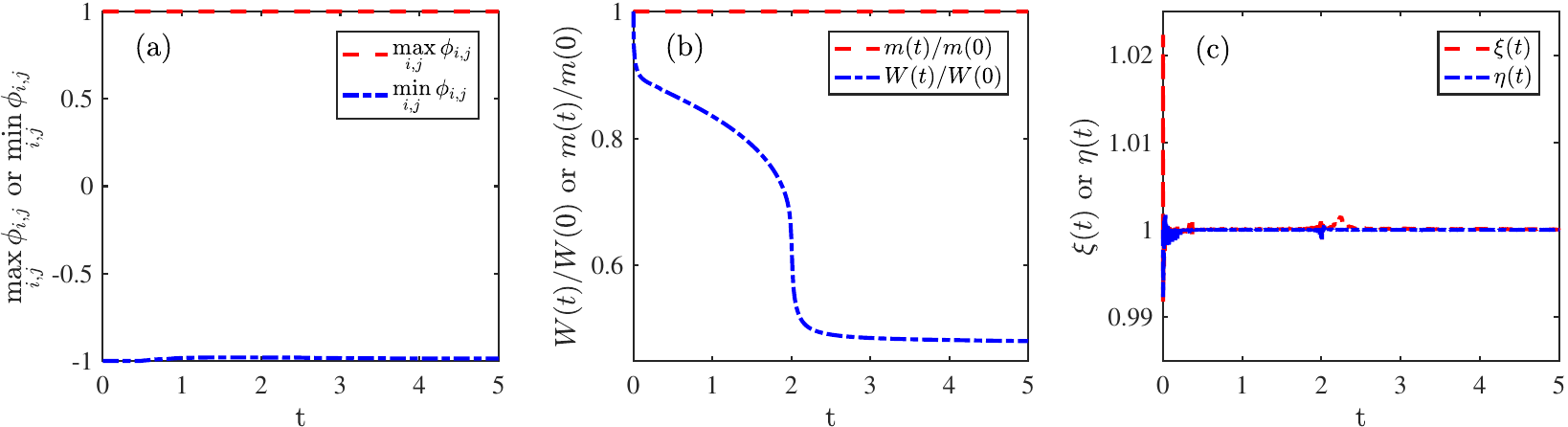}
\caption{
Time histories of (a) the extrema of $\phi_{i,j}$, (b) normalized energy and mass, and (c) the Lagrange multipliers $\xi(t)$ and $\eta(t)$ under the polynomial potential.}\label{figsdiffusion3}
\end{figure}

\begin{table}[!htp]
\renewcommand\arraystretch{1.5}
\centering
\caption{
Relative total film-area change $\delta S=|S(t)-S(0)|/S(0)$ at $t=5$ for the three-island configuration.}\label{tabS}
\tabcolsep 0.31in
\begin{tabular}[c]
{cccccc} \hline
&$\theta=0.3$   &$\theta=0.4$   &$\theta=0.5$   &$\theta=0.6$   &$F_{pol}(\phi)$\\   \hline
$\delta S$    &$0.1578\%$   &$0.3023\%$   &$0.3252\%$   &$1.4090\%$   &$1.8791\%$\\   \hline
\end{tabular}
\end{table}

\subsection{Contact line mobility}

The preceding tests focus on the influence of the potential; we now turn to the boundary relaxation that couples the film to the substrate. Its rate is set by $\kappa$, the contact line mobility, while $\theta_s$ prescribes the Young angle toward which the contact line relaxes. We compare $\kappa=10^3,2\times10^3,5\times10^3,10^4$ for the logarithmic and polynomial potentials. The apparent contact angle $\theta_d$, measured from the diffuse interface, retains a small offset from $\theta_s$ in the reported stationary states. Finite interface width, discretization, and angle extraction can all contribute to this offset.

Figure~\ref{figka1} shows relaxation from an initial angle $\theta_d=\pi/2$. Larger $\kappa$ gives a faster initial increase in the angle. The curves for $\kappa=5\times10^3$ and $10^4$ are close for the tested parameters, suggesting limited sensitivity to further increases in mobility in this regime. This trend is consistent with the IEQ computations in~\cite{Huang19b}.

The relaxation curves describe the approach to a stationary state; Table~\ref{tabka1} examines the apparent angle reached in that state. For this test, lower logarithmic-potential temperatures give angles closer to the prescribed value. At $\theta=0.6948$, the logarithmic and polynomial values are similar. This is an empirical comparison: the shrinkage-prefactor crossover in Remark~\ref{remdr} does not establish an identity between their contact angles. Figures~\ref{figka3} and~\ref{figka4} show bounded phase values, decreasing energy, and conserved mass to the resolution displayed.

\begin{figure}[!htp]
\centering
\includegraphics[width=\textwidth]{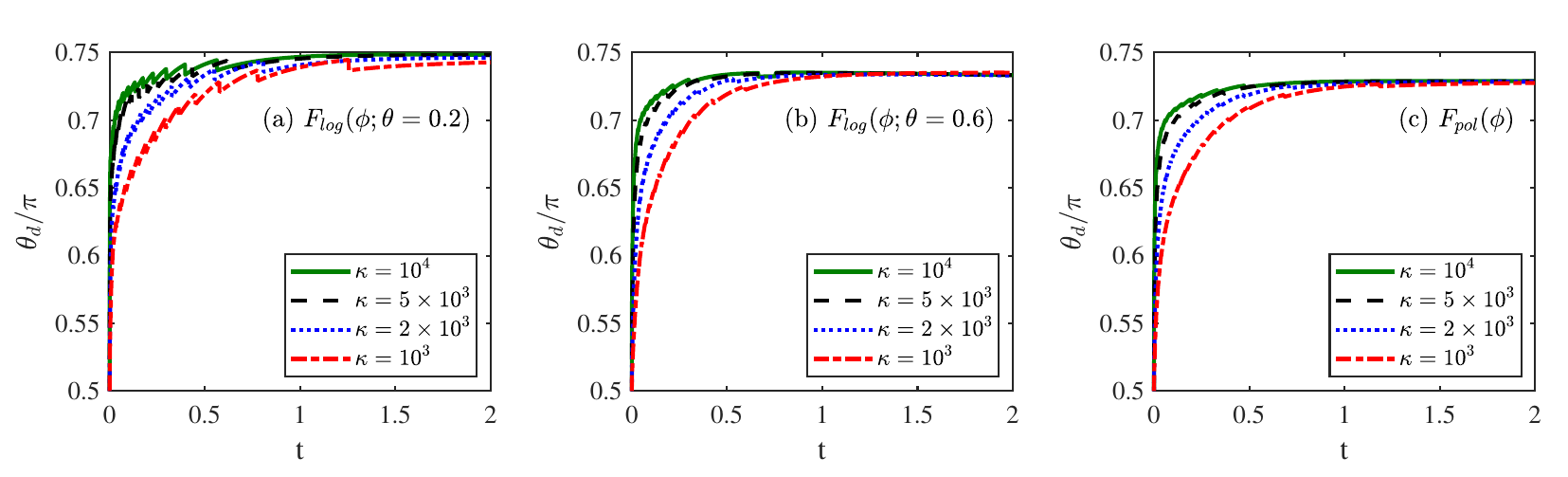}
\caption{
Dynamic contact angle for the indicated mobilities and potentials. The vertical axis shows $\theta_d/\pi$.}\label{figka1}
\end{figure}

\begin{figure}[!htp]
\centering
\includegraphics[width=\textwidth]{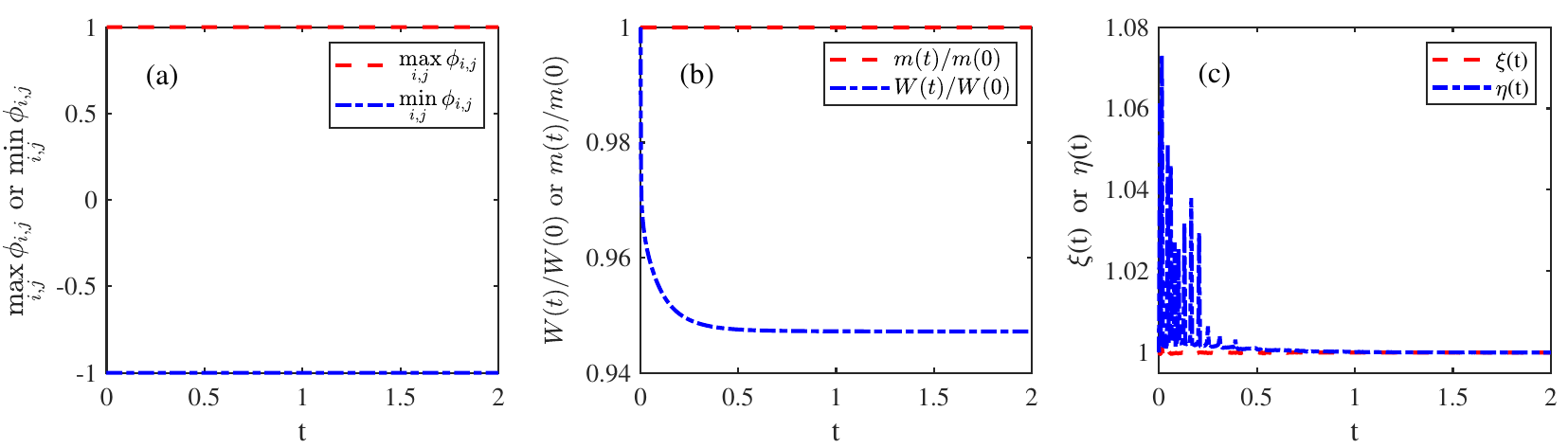}
\caption{
Time histories of (a) the extrema of $\phi_{i,j}$, (b) normalized energy and mass, and (c) the Lagrange multipliers $\xi(t)$ and $\eta(t)$ under the logarithmic potential at $\theta=0.2$.}\label{figka3}
\end{figure}

\begin{figure}[!htp]
\centering
\includegraphics[width=\textwidth]{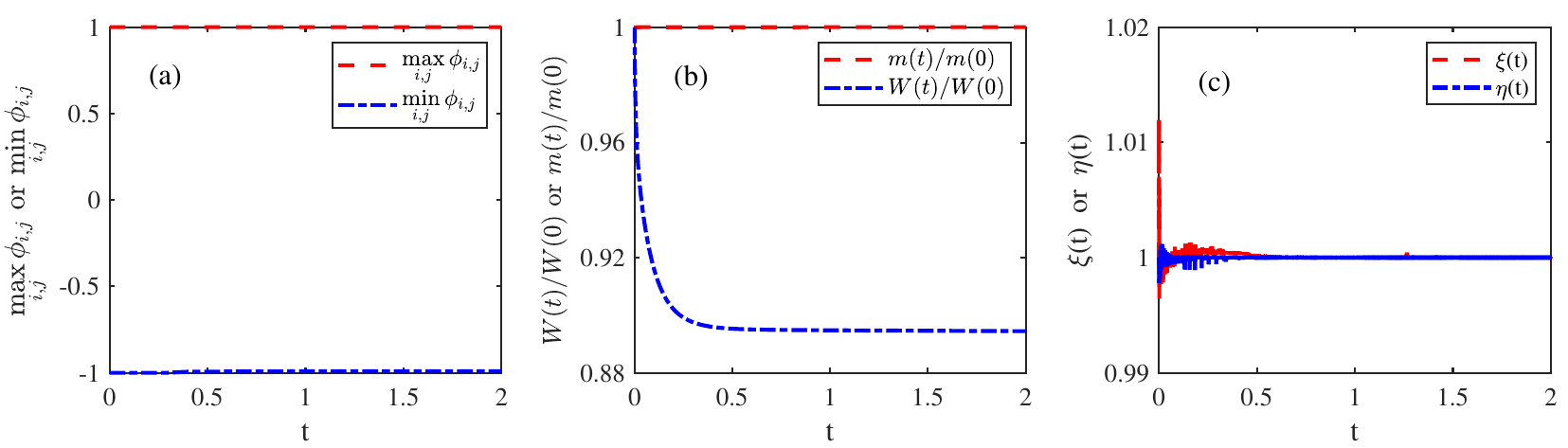}
\caption{
Time histories of (a) the extrema of $\phi_{i,j}$, (b) normalized energy and mass, and (c) the Lagrange multipliers $\xi(t)$ and $\eta(t)$ under the polynomial potential.}\label{figka4}
\end{figure}

\begin{table}[!htp]
\renewcommand\arraystretch{1.5}
\centering
\caption{
Apparent contact angle $\theta_d$ at numerical equilibrium for different potentials. All angles are in radians.}\label{tabka1}
\tabcolsep 0.21in
\begin{tabular}[c]
{ccccccc} \hline
$\theta_{s}$ &$\theta=0.2$ &$\theta=0.4$ &$\theta=0.6$ &$\theta=0.6948$ &$\theta=0.8$ &$F_{pol}(\phi)$ \\   \hline
$\frac{3\pi}{4}\approx 2.3562$   &2.3506   &2.3307   &2.3040  &2.2908  &2.2773   &2.2898\\  \hline
\end{tabular}
\end{table}

\FloatBarrier
\subsection{Wettability-dependent evolution}

Having examined the relaxation rate at a prescribed Young angle, we next vary that angle to study how substrate wettability changes the evolving film shape. We follow an initially rectangular island at $\theta_s=0,\pi/4,\pi/3,\pi/2,3\pi/4,\pi$, using the logarithmic potential with $\theta=0.2$. The endpoint values represent the complete-wetting and complete-dewetting limits of the prescribed wall energy.

Figure~\ref{figths} shows spreading at $\theta_s=0$, an approximately circular island at $\theta_s=\pi$, and intermediate cap-like shapes at the other angles. Table~\ref{tabD} reports relative area changes at $t=6.5$, decreasing from $0.5527\%$ to $0.0153\%$ across these cases. For intermediate angles, this trend is qualitatively consistent with the reduced contraction predicted by \eqref{eqt7} at larger $\theta_s$ when the other geometric parameters are fixed. The table is not a direct test of the inverse-angle law: the evolving geometries differ, and the estimate assumes $0<\theta_s<\pi$, excluding both endpoints.

\begin{figure}[!htp]
\centering
\includegraphics[width=\textwidth]{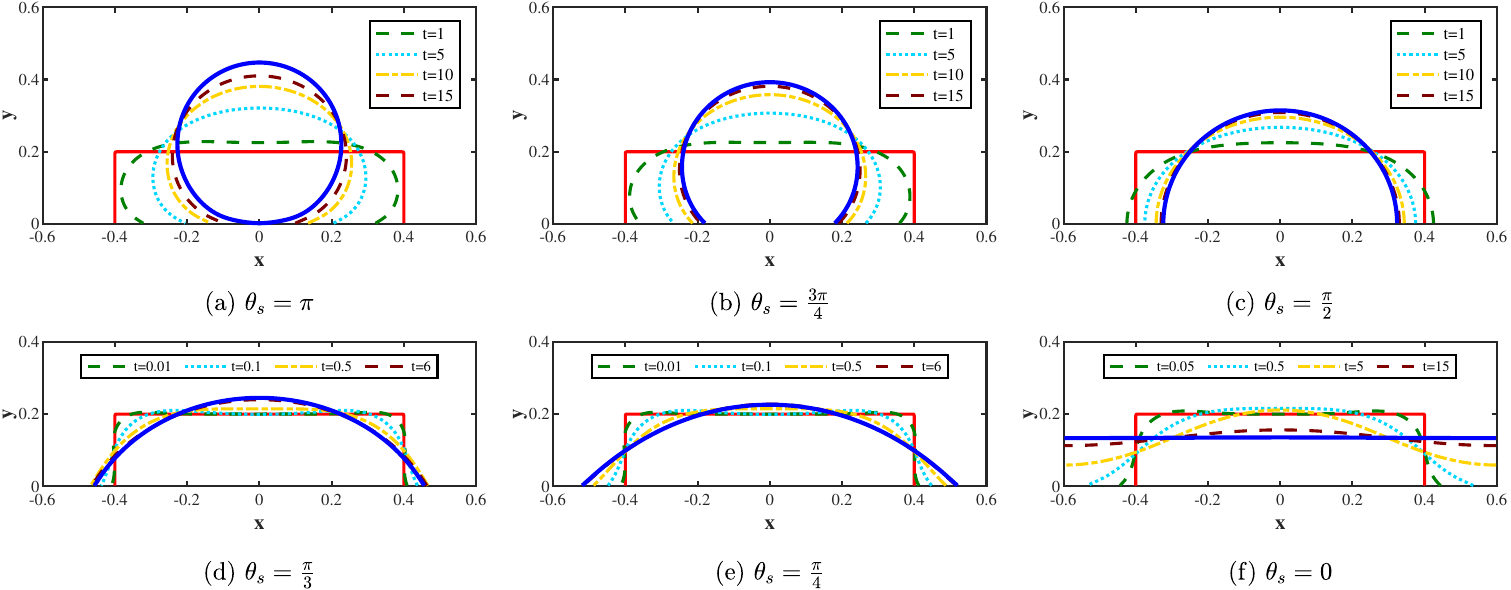}
\caption{
Evolution of an initially rectangular film for six prescribed Young angles. Colors and line styles indicate the states and times shown in each panel.}\label{figths}
\end{figure}

\begin{table}[!htp]
\renewcommand\arraystretch{1.5}
\centering
\caption{
Relative total film-area change $\delta S=|S(t)-S(0)|/S(0)$ at $t=6.5$ for different prescribed Young angles.}\label{tabD}
\tabcolsep 0.23in
\begin{tabular}[c]
{ccccccc} \hline
&$\theta_{s}=0$   &$\theta_{s}=\frac{\pi}{4}$   &$\theta_{s}=\frac{\pi}{3}$   &$\theta_{s}=\frac{\pi}{2}$   &$\theta_{s}=\frac{3\pi}{4}$   &$\theta_{s}=\pi$\\   \hline
$\delta S$   &$0.5527\%$   &$0.2591\%$   &$0.2354\%$   &$0.0453\%$   &$0.0168\%$   &$0.0153\%$\\   \hline
\end{tabular}
\end{table}

\subsection{Pinch-off}

The preceding examples describe relaxation toward individual island shapes; an elongated film also raises the question of whether the evolving interface remains connected. To examine how the potential influences this topological outcome, we compare the evolution of a rectangular film initially occupying $[-1,1]\times[0,0.1]$, with aspect ratio $20$. We use the logarithmic potential at $\theta=0.3$ and the polynomial potential, with $\varepsilon=0.03$ and $\theta_s=3\pi/4$ in both cases.

\begin{figure}[!htbp]
\centering
\includegraphics[width=12cm]{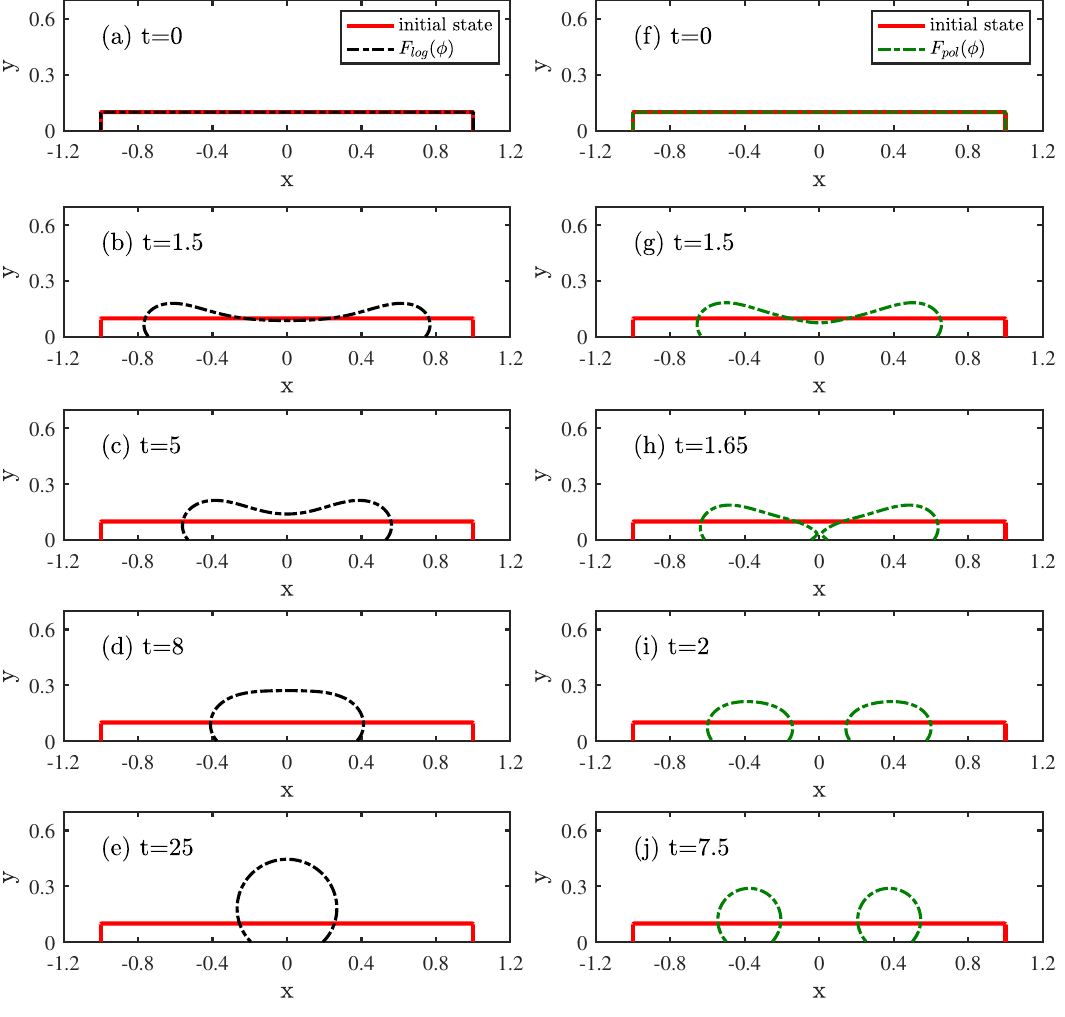}
\caption{
Evolution of a film initially occupying $[-1,1]\times[0,0.1]$: (a)--(e) logarithmic potential with $\theta=0.3$; (f)--(j) polynomial potential. Both cases use $\theta_s=3\pi/4$ and $\varepsilon=0.03$. Snapshot times differ between the two columns.}\label{figpinchoff}
\end{figure}

With the polynomial potential, the film separates into two islands and the reported stationary configuration has a relative area change of $17.5386\%$. With the logarithmic potential at $\theta=0.3$, it remains connected over the simulated interval and relaxes toward a single rounded island, with an area change of $0.0767\%$ at $t=25$. Figure~\ref{figpinchoff} illustrates this strong sensitivity to the potential and the substantially smaller area variation in the logarithmic case. This sensitivity makes the choice of potential consequential for topology as well as area loss, although the comparison alone does not identify which breakup history represents the sharp-interface limit. Surface diffusion can itself produce pinch-off~\cite{Wong00}; determining whether the polynomial-potential breakup in this configuration is spurious requires refinement studies or a matching sharp-interface benchmark.

Figures~\ref{figspinchoff2} and~\ref{figspinchoff3} show bounded phase values, decreasing energy, and conserved mass to the plotted resolution for both potentials. The multipliers settle near $1$ after transients, including a larger initial excursion of the wall multiplier in the logarithmic case. Thus, similar structural diagnostics can accompany different morphological outcomes; structure preservation alone does not establish physical fidelity.

\begin{figure}[H]
\centering
\includegraphics[width=\textwidth]{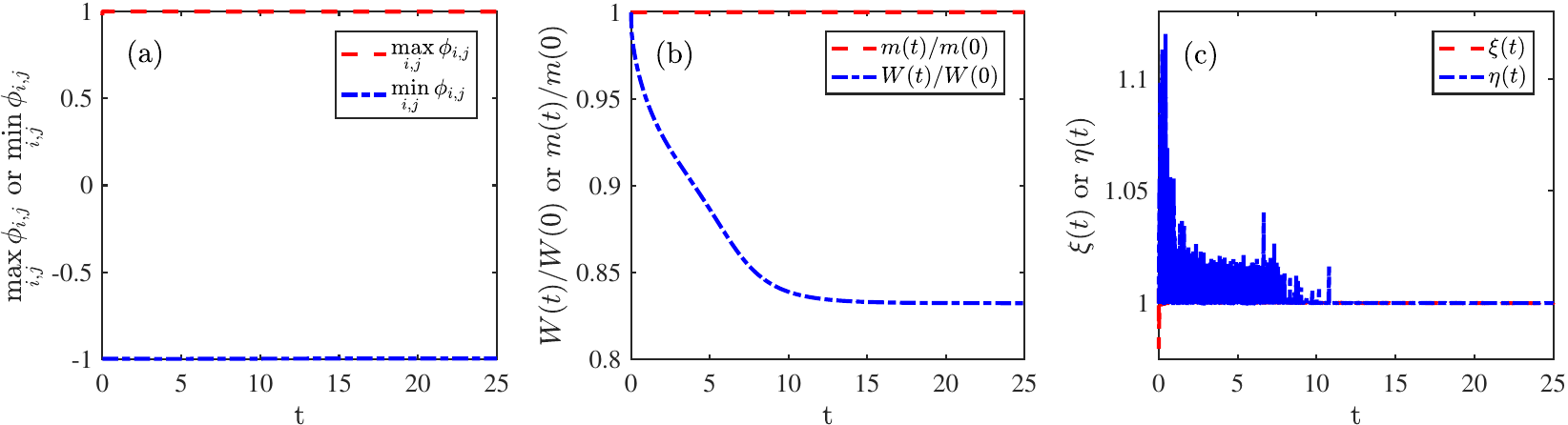}
\caption{
Time histories of (a) the extrema of $\phi_{i,j}$, (b) normalized energy and mass, and (c) the Lagrange multipliers $\xi(t)$ and $\eta(t)$ under the logarithmic potential at $\theta=0.3$.}\label{figspinchoff3}
\end{figure}

\begin{figure}[H]
\centering
\includegraphics[width=\textwidth]{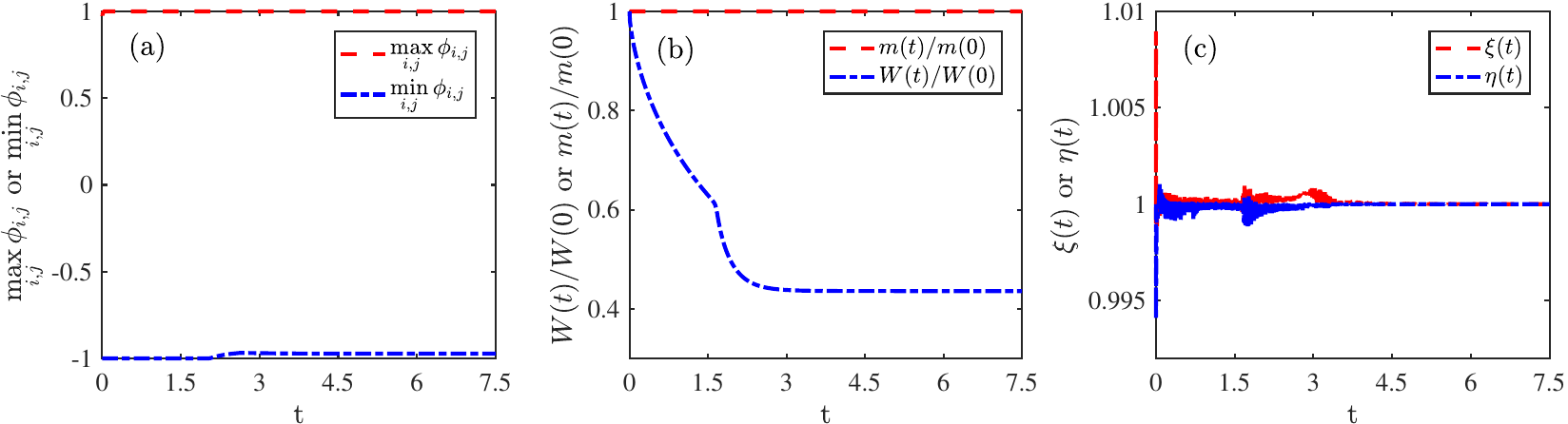}
\caption{
Time histories of (a) the extrema of $\phi_{i,j}$, (b) normalized energy and mass, and (c) the Lagrange multipliers $\xi(t)$ and $\eta(t)$ under the polynomial potential.}\label{figspinchoff2}
\end{figure}

\FloatBarrier
\section{Conclusion}
\label{seccon}

We have addressed the coupling of bulk and wall energies in phase-field simulations of solid-state dewetting. The proposed upwind Lagrange multiplier scheme combines a logarithmic Flory--Huggins potential with dynamic contact line conditions. The construction couples the bulk and wall energy balances and, for solutions of the discrete equations, preserves strict phase-field bounds, mass, and energy dissipation. Sequential row and column updates retain these properties while reducing the size of the nonlinear systems.

Alongside these discrete guarantees, the shrinkage analysis and numerical experiments assess finite-width effects that structure preservation alone cannot resolve. The leading-order estimate relates radius contraction to the potential, interface width, domain size, and contact angle within a specified circular-segment ansatz. The numerical results support its low-temperature predictions and show increasing deviations toward the upper end of the tested temperature range. The remaining experiments demonstrate reduced area variation and inter-island coarsening at low temperature, as well as sensitivity to contact line mobility and wettability. The elongated-film comparison shows that the potential can alter breakup behavior even when both computations preserve the same discrete structural properties.

The different breakup histories make refinement studies and matching sharp-interface benchmarks a necessary next step in assessing physical fidelity. Complementary numerical analysis should establish nonlinear solvability and convergence and quantify mesh and time-step errors. Extensions to anisotropic surface energies and three-dimensional geometries would broaden the model's applicability and allow the computational benefits of splitting to be assessed in more demanding settings.

\begin{acknowledgement}

The numerical computations were performed at the Supercomputing Center of Wuhan University.
\end{acknowledgement}
\medskip
\noindent{\textbf{Funding}} ~ This work was partially supported by the National Natural Science Foundation of China (Nos. 12001210, 12131010, 12301558), the Natural Science Foundation of Henan Province (No. 252300420308) and the Key Scientific Research Project of Universities in Henan Province (No. 27A110002).

\medskip
\noindent{\textbf{Data Availability}} ~
The code used in this study is available from the corresponding author upon reasonable request.

\section*{Declarations}
{\textbf{Conflict of interest}} ~The authors declare no competing interests.



\end{document}